\documentclass{iciam}

\usepackage{graphicx}    
\usepackage{color}
\usepackage{array}
\usepackage{bm,euscript,ulem} 
\usepackage{amstext,amssymb,mathtools,esint,tabularx}
\mathtoolsset{showonlyrefs} 
\usepackage{tikz,pgfplots,pgfplotstable} 
\pgfplotsset{compat=newest}
\usetikzlibrary{calc,3d,arrows,patterns,decorations}
\tikzset{>=stealth'}
\usetikzlibrary{positioning} 
\usepackage[abs]{overpic}
\usepackage[outline]{contour}
\contourlength{1.6pt}
\usetikzlibrary{decorations.text}
\tikzset{text deco/.style={postaction={decorate, decoration={text along path,#1}}}}

\contact[ulrich.ruede@fau.de]{Universit\"at Erlangen--N\"urnberg, Cauerstrasse 11, 91058 Erlangen, Germany}
\contact[wohlmuth@ma.tum.de]{Technische Universit\"at M\"unchen, Boltzmannstrasse 3, 85748 Garching, Germany}

\newtheorem{theorem}{Theorem}[section]

\theoremstyle{definition}

\def\div{\mathop{\mathrm{div}}\nolimits}

\newcommand{\ub}{\mathbf{u}}

\newcommand{\vb}{\mathbf{v}}
\newcommand{\Ub}{\mathbf{U}}
\newcommand{\Vb}{\mathbf{V}}
\newcommand{\wb}{\mathbf{w}}
\newcommand{\fb}{\mathbf{f}}
\newcommand{\fst}{f_\text{st}}
\newcommand{\cst}{c}
\newcommand{\fstb}{\mathbf{f}_\text{st}}

\newcommand{\uh}{\ub_h}
\newcommand{\ph}{p_h}
\newcommand{\vh}{\vb_h}
\newcommand{\wh}{\wb_h}
\newcommand{\qh}{q_h}
\newcommand{\Vh}{\Vb_h}
\newcommand{\Qh}{Q_h}

\newcommand{\hmod}[1]{{#1}^{\text{m}}_h}

\newcommand{\phm}{\hmod{p}}
\newcommand{\uhm}{\hmod{\ub}}

\newcommand{\sui}{\mathbf{s}^u_{i}}
\newcommand{\spi}{s^p_{i}}

\newcommand{\su}{\mathbf{s}^u}

\newcommand{\gam}{\boldsymbol{\gamma}}

\newcommand{\piso}{$P_1\text{--iso}P_2/P_1\,$}

\title[Solution Techniques for the Stokes System]
{Solution Techniques for the Stokes System:\\
\large A~priori and a~posteriori modifications, resilient algorithms}

\author[M.~Huber, L.~John, P.~Pustejovska, U.~R\"ude, C.~Waluga, B.~Wohlmuth]{Markus Huber, Lorenz John, Petra Pustejovska, Ulrich R\"ude,\\ Christian Waluga, Barbara Wohlmuth\thanks{
This work was supported by the German Research Foundation (DFG) through the Priority Programme 1648 “Software for Exascale Computing” (SPPEXA).}}

\begin{document}
\normalem

\begin{abstract}
This article proposes modifications
to standard low order finite element approximations
of the Stokes system
with the goal
of improving both the approximation quality and the parallel algebraic
solution process.
Different from standard finite element techniques, we do
not modify or enrich the approximation spaces but modify the operator itself
to ensure fundamental physical properties such as
mass and energy conservation.
Special local a~priori correction techniques at re-entrant corners
lead to an improved representation of the energy in the discrete system
and can suppress the global pollution effect.
Local mass conservation can be achieved by an a~posteriori
correction to the finite element flux.
This avoids artifacts in coupled multi-physics
transport problems.
Finally, hardware failures in large supercomputers may lead to a loss of
data in solution subdomains.
Within parallel multigrid, this
can be compensated by 
the accelerated solution of local subproblems.
These resilient algorithms will gain
importance on future extreme scale computing systems.
\end{abstract}



\maketitle

\section{Introduction}

%

Advances in numerical methods together with progress in computer technology 
enables the modeling and simulation of ever more complex physical
phenomena with increasingly higher accuracy.
Cost aware numerical simulation,
embedded in the wider field of computational science and engineering
has thus grown to be a fundamental pillar for 
all quantitative sciences.

The past decades have seen the development of an extensive theoretical
foundation of finite element (FE) approximation and for optimal solution algorithms.
In the context of simulation technology,
the success of this research must be measured by
which accuracy can be achieved at which cost.
This naturally leads to the question how both accuracy and cost can be quantified.
However, an answer to these seemingly trivial and obvious 
questions is more difficult than it may appear at first sight.
In particular, the cost metrics have undergone a dramatic change 
due to the advent of ever more complex, heterogeneous and hierarchical parallel computer systems.
For the development of  numerical methods that will be used in large scale simulations,
it has become important to acknowledge that non-parallel computers have ceased to exist.
Equally important, we will illustrate that even low cost computers can
solve huge problems, provided the models, algorithms, and software
have been developed to exploit their full power.

In this article, we will present new research results that explore
the duality of accuracy versus cost. We will use
the Stokes system for viscous incompressible flow as a guiding example.
The Stokes problem in the 
polyhedral domain 
$\Omega\subset \mathbb{R}^d$, $d=2,3$,
with homogeneous Dirichlet boundary conditions
is given by
\begin{equation}\label{stokes}
	\begin{alignedat}{3}
	-\Delta \mathbf{u} +\nabla p & = \mathbf{f} &\quad& \text{in }\Omega,\\
		\mathrm{div}\, \mathbf{u} & = 0 &&\text{in }\Omega,\\
		\mathbf{u} &= \mathbf{0} &&\text{on } \partial\Omega,\\
	\end{alignedat}
\end{equation}
where $\mathbf{u}$ denotes the velocity field, $p$ the pressure,
 and $\mathbf{f}$
represents a given force term.

The methods proposed below are corrections that compensate known numerical 
deficiencies with minimal additional cost.
In the case of re-entrant corners, the local singularity leads to a  global deterioration
of accuracy.
Rather than enriching the approximation space to better capture the singular behavior,
we propose a strictly local a~priori modification of the energy inner product,
showing that once the energy is correctly represented, the global error pollution
will be eliminated. 
This avoids an enrichment of the FE approximation space by, e.g., local mesh refinement
that would lead to more complex data structures, a more involved
load balancing strategy and an increased communication in solution algorithms.
%
The energy correction, in contrast, requires only the change of a few coefficients of the stiffness
matrix, and thus does not lead to a significant increase in cost.

Analogously, we will present an a~posteriori correction technique that equips simple
stabilized equal order elements with the much-desired local mass conservation.
This again avoids the use of more complex discretizations
that would 
in turn lead to an increased cost in the solver and the parallel communication.

The solvers presented here rely on the multigrid principle
combined suitably with Krylov space acceleration.
We will demonstrate that these methods do not only have optimal complexity,
but that they can be implemented
with excellent parallel performance.
Here we go beyond the classical notion of scalability
which is only a necessary but no sufficient prerequisite for a parallel solver
being fast. Having a locally conservative scheme and a fast solver allow us to do large-scale  long-term  simulation runs that
 arise from simplified Earth mantle convection problems.

On future supercomputers, a fail-safe operation of the complete system
may not be guaranteed. In this case, some components may fail and the corresponding
part of the  running computation may be lost. Then,
it is  attractive to use algorithms that can compensate for a partial loss of data.
This algorithmically based fault tolerance is 
an intrinsic correction technique that compensates for large local
errors that are naturally introduced when a processor fails.

The rest of this article is structured as follows: In Sect.~\ref{sec:discr}, we 
present the standard notation and 
recall the stabilized equal-order finite element discretization.
Then, in Sect.~\ref{sec:ec}, we introduce the method of energy-corrections 
and demonstrate 
its effectiveness for 
simple examples in two dimensions.  
Parallel multigrid solvers for large scale computations
and 
the new fault tolerant algorithms are 
presented in Sect.~\ref{sec:solver}.
Finally, we discuss 
local mass-conservative correction schemes in Sect.~\ref{sec:mc}
including application to problems arising from geophysics.

\section{Preliminaries  and finite element discretization}\label{sec:discr}
Through this work, we assume $\Omega$ to be an open, bounded, simply connected polyhedral domain.  We  employ standard notation, i.e., $H^k(\Omega)$, $\|\cdot\|_k$, $k\in\mathbb{N}$, denote the Sobolev spaces of functions having square-integrable generalized derivatives up to order $k$, and the standard Sobolev norm, respectively. The case of $k=0$ coincides with the Lebesgue space $L^2(\Omega)$. Additionally, we denote the $H^1(\Omega)$-space with vanishing trace on $\partial\Omega$ by $H^1_0(\Omega)$, and the space of square integrable functions with vanishing mean by $L^2_0(\Omega)$. We also write $X^d$ for a space of vector valued functions which components belong to the function space $X$. 


The main focus of this paper are continuous low equal-order velocity and pressure discretizations. For conforming and shape-regular partitions $\mathcal{T}_h$ of $\Omega$ into simplicial elements, we assume finite element spaces $\Vh=V_h^d\cap H^1_0(\Omega)^d$ and $Q_h=V_h\cap L^2_0(\Omega)$, where $V_h \coloneqq \left\{v_h \in H^1(\Omega): v_h|_T \in P_1(T) \ \forall T \in \mathcal{T}_{h}\right\}$.
It is well-known, that such a pair violates the discrete inf-sup condition. Thus, we work in the following, unless not further specified,  with a stabilized finite element scheme. Namely, let us define $a(\ub,\vb) \coloneqq \int_\Omega \nabla\ub:\nabla\vb\, dx$ and $b(\vb,q) \coloneqq - \int_\Omega q\, \div \vb\, dx$, $\vb,\ub\in H^1(\Omega)$, $q\in L^2(\Omega)$, for which the discrete variational formulation of \eqref{stokes} reads: Find $(\uh, \ph) \in \Vh\times\Qh$ such that
\begin{equation}\label{vastab}
	\begin{alignedat}{3}
		&a(\uh,\vh) + b(\uh,\ph) &&=\langle \fb,\vh\rangle \qquad &&\forall \vh\in \Vh,\\
		&b(\uh,\qh) - \cst(\qh, \ph) &&=\langle \fst,\qh\rangle &&\forall \qh\in \Qh,
	\end{alignedat}
\end{equation}
where the additional	 stabilization terms $\cst(\cdot,\cdot)$ and $\langle \fst,\cdot\rangle$ are specified by the particular stabilization method. In many cases, we shall assume the standard pressure stabilization for which $\langle \fst,\cdot\rangle = 0$ and
\begin{align}
	\cst(\qh, \ph) \coloneqq \sum_{\mathclap{T \in \mathcal{T}_h}} c_T(p_h,q_h),\quad c_T(p_h, q_h) \coloneqq \delta_T h_T^2 \int_T \nabla \ph \cdot \nabla \qh\,  dx,
\end{align}
where $\delta_T > 0$ is a certain stabilization parameter specified later, or the PSPG stabilization where additionally $\langle \fst,\qh\rangle \coloneqq -\sum_{T \in \mathcal{T}_h} \delta_T\, h_T^2 \int_T \fb \cdot \nabla \qh\,  dx$. For details, see, e.g., \cite{Brezzi2013,brezzi-douglas_1988}.

Throughout the paper, for simplicity of notation, we use the symbols $\lesssim$ and $\gtrsim$ for $\leqslant \text{const.}$ and $\geqslant \text{const.}$, respectively, where $\text{const.}$ represents a generic positive constant independent of the mesh size $h$. 

\section{A priori energy correction}\label{sec:ec}
In the presence of 
re-entrant corners, standard numerical schemes 
on uniformly refined meshes
suffer from a pronounced 
deterioration of the global convergence rates.
Here, we focus on the recovery of the optimal convergence rates by a local a~priori modification of a standard finite element scheme in two dimensions.
For simplicity, we restrict ourselves to the case of a single re-entrant corner. 
Before we present our method, let us state here the necessary regularity results for the solution of \eqref{stokes}, together with its main properties. 
For that, some weighted Lebesgue spaces with weight as a distance function $r$ to the non-convex domain-vertex\footnote{In what follows, such a point we simply call singular point.} $x_0$, see, e.g., \cite{adams-fournier_book,Kondratjev1967,kufner_1985}, play an important role. Namely, we define for $\beta\in\mathbb{R}$ and $r(x) = \|x-x_0\|_{l^2}$:
\begin{align}
	L^2_\beta(\Omega)\coloneqq \Big\{f\in L^1_{\text{loc}}(\Omega): \textstyle{\int_\Omega} |r^{\beta} f|^2 dx<\infty \Big\}, \quad \|f\|^2_{0,\beta} \coloneqq \left(\int_{\Omega} |r^\beta f|^2\,dx \right).
\end{align}
The weighted Sobolev space of Kondratev-type of order $k$, are then given by
 $H^k_\beta(\Omega)\coloneqq\left\{f\in L^2_\beta(\Omega): r^{\beta-k+|\alpha|} D^{\alpha} f \in L^{2}(\Omega), |\alpha|\leqslant k \right\}$, where $\alpha=(\alpha_1,\alpha_2)$ is a multi-index with $|\alpha|=\alpha_1+\alpha_2$, and $D^\alpha$ represents the $\alpha$-th generalized derivative, see, e.g., \cite{Kondratjev1967}. The space $H^{k}_\beta(\Omega)$ is equipped with the norm $\|f\|^2_{k,\beta}\coloneqq \sum_{|\alpha|\leqslant k} \| r^{\beta-k+|\alpha|} D^{\alpha} f \|^2_0$. The spaces $L^2_\beta(\Omega)$ and $H^k_\beta(\Omega)$ are Hilbert spaces.
\subsection{Singularities for Stokes flow at re-entrant corners}\label{sec:singularities}
A standard spectral analysis applied on the Dirichlet system \eqref{stokes} leads to a specification of the eigenvalues and (generalized) eigenvectors of the Stokes operator 
that can be used for a local additive composition of the solution, see, e.g., \cite{Dauge1989, Kozlov2001, Markl2008} and the original work \cite{Kondratjev1967}. 
In particular for the Dirichlet problem, the eigenvalues are specified by the roots of 
\begin{align}\label{root}
	\lambda_i^2 \sin^2 (\omega) - \sin^2(\lambda_i\omega)=0, \qquad \lambda_i\neq 0, \quad \lambda_i\in\mathbb{C}. 
\end{align}
Note, that the solutions of \eqref{root} may be complex
and possibly  have higher multiplicity, at most two, 
see Fig.~\ref{F:roots}.

\begin{figure}[ht]
	\centering
	\includegraphics[width=0.75\textwidth]{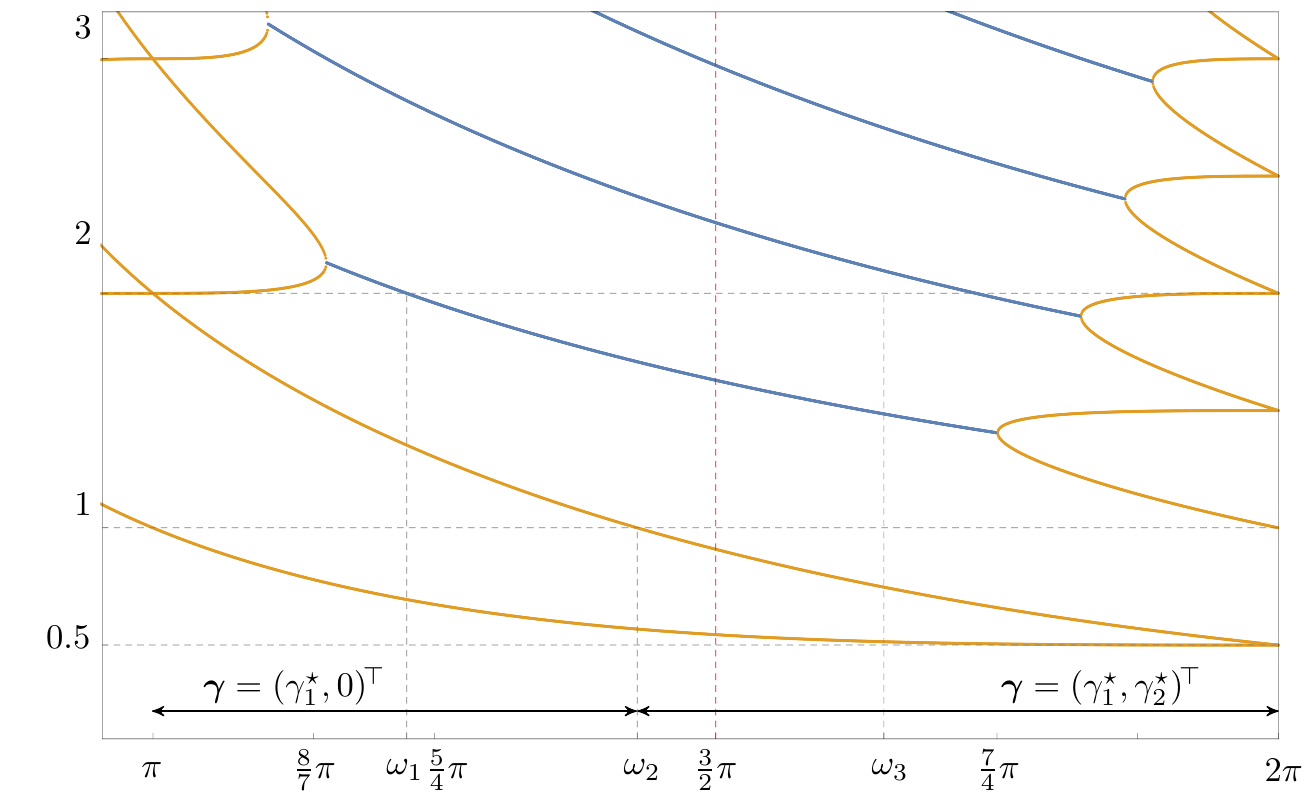}
	\caption{Distribution of the real part of the eigenvalues $\lambda$ with respect to the angle $\omega$ of the re-entrant corner. The blue lines represent complex eigenvalues, the orange ones represent the real eigenvalues. The points of bifurcations, also the point where  $\lambda_2 =1$, have increased multiplicity. In the horizontal axis, $\omega_i$, $i=1,2,3$, represent the following cases: $\omega_1$ is defined such that $\lambda_1 + \lambda_2 =2$, i.e., $\omega_1 = 1.22552...\pi$, $\omega_2$ is the unique angle such that $\lambda_2 =1$, i.e., $\omega_2=1.430296...\pi$, and $\omega_3$ represents the angle for which $\lambda_1+\mathrm{Re}(\lambda_3)=2$, i.e., $\omega_3=1.64941...\pi$. In the plot, there are also included two $\omega-$intervals, $(\pi,\omega_2)$ and $(\omega_2,2\pi)$, see Theorem~\ref{thmmain} and the discussion after.}
	\label{F:roots}
\end{figure}
%

This is significant, since 
these roots constrain the regularity of the particular singular functions. 
The eigenvectors are then, in the polar coordinate-system $(r,\theta)$, 
log-polynomial functions in $r$ 
multiplied with smooth functions in $\theta$. Explicitly, for each $\lambda_i$ with  geometric and algebraic multiplicity $I_i$ and $\kappa_{i,j}$, respectively, the singular functions 
for the velocity and pressure part of the solution are of the form
\begin{align}\label{sinfce}
	\hspace{-0.2cm}\su_{i,j,k}=r^{\lambda_i}\sum_{l=0}^{k}(\ln r)^l \boldsymbol{\varphi}^{(i)}_{j,k-l}(\theta), \quad
	s^p _{i,j,k}= r^{\lambda_i-1}\sum_{l=0}^{k}(\ln r)^l \xi^{(i)}_{j,k-l}(\theta),
\end{align}
$k\in\{0,\kappa_{i,j}-1\}$.
Under the assumption that the 
right-hand side has a higher regularity,
i.e., $\fb\in L^{2}_{-\alpha}(\Omega)^2$ for some $\alpha>1-\lambda$, 
together with the assumption that $\lambda_i$ are simple roots for all $i\in[1,N]$, 
there holds the following expansion of the solution near the singular point (for simplicity we skip the $j,k$ indices):
\begin{align}\label{updec}
 	\ub= \sum_{i=1}^{N} c_i\ \su_i + \mathbf{U}, \qquad p= \sum_{i=1}^{N} c_{i}\ s^p _{i} + P,
\end{align}
where $(\Ub,P)\in H^2_{-\alpha}(\Omega)^2\times H^1_{-\alpha}(\Omega)$ is called a smooth remainder. For our needs, we will require a decomposition with two singular functions, i.e., $N=2$. The assumption on simple multiplicity of the eigenvalues is in that case 
violated only for one angle $\omega_2$ in the whole $(\pi,2\pi)$ range.

It is clear from the structure of the singular functions \eqref{sinfce} that the integrability and boundedness of the gradients is limited which results in a reduced regularity of the solution in domains with re-entrant corners. On the other hand, the special polynomial structure in $r$ of the singularities can be well-handled in the framework of the Kondratev-type Sobolev spaces. Hence, one can obtain a generalized regularity result 
\begin{align}\label{AESt}
	\|\ub\|_{2,\beta} + \|p\|_{1,\beta} \lesssim \|\fb\|_{0,\beta},
\end{align}
see, e.g., \cite{Guo2006}. The explicit form of the singular parts of the solution $(\ub,p)$, decomposition \eqref{updec}, a~priori bounds following from the structure of the singularities and the smooth remainder, together with regularity \eqref{AESt} are then used in the formulation and the proof of the convergence behavior of the energy-corrected methods, as they are stated below.
\smallskip

\subsection{Pollution effect}
It is well-known that the influence of the non-convex corners significantly reduces the global accuracy of standard numerical methods. More precisely, standard Galerkin approximations on a sequence of uniformly refined meshes 
of the Stokes problem are, in general, of the order:
\begin{align}\label{pol}
	\|\nabla(\ub-\uh)\|_0 + h^{-\lambda_1}\|\ub-\uh\|_0+\|p-\ph\|_0=\mathcal{O}(h^{\lambda_1}),
\end{align}

\noindent cf. \cite{Arnold1984,Blum1990,Girault1986}. This global effect, which is commonly referred to as \emph{pollution} cannot be cured by considering different error norms. Also, it is exhibited by all standard (piecewise polynomial) approximations, independently of the approximation order. For the illustration of the pollution effect, we present the errors and convergence rates for the stable Taylor--Hood element in Tab.~\ref{T:polTH}, where we compare a typical global pollution in the standard approximation with the non-polluted interpolation; cf. also Fig.~\ref{F:polTH} for some illustration.
\begin{table}[ht!]
	\setlength\tabcolsep{3pt}
	\centering
	\small{\begin{tabular}{| c | c | c | c || c | c | c | c |}
	\hline &&&&&&& \\[-1.0em]
	$\| \ub - \uh \|_0$ & eoc & $\| p - p_h \|_0$ & eoc & $\| \ub - \uh \|_{0,\alpha}$ & eoc & $\| p - p_h \|_{0,\alpha}$ & eoc \\ \hline\hline
	6.24297e--02 & --   &  3.76799e--01 & --   &  2.94081e--02  & --   &  4.73148e--01 &  --  \\
	3.03536e--02 & 1.04 &  2.00709e--01 & 0.91 &  9.38624e--03  & 1.65 &  1.72502e--01 &  1.46\\
	1.20753e--02 & 1.33 &  1.27266e--01 & 0.66 &  3.21263e--03  & 1.55 &  7.58580e--02 &  1.19\\
	4.90503e--03 & 1.30 &  8.06119e--02 & 0.66 &  1.36660e--03  & 1.23 &  3.50865e--02 &  1.11\\
	2.06841e--03 & 1.25 &  5.23079e--02 & 0.62 &  6.20441e--04  & 1.14 &  1.65028e--02 &  1.09\\
	9.04139e--04 & 1.19 &  3.47209e--02 & 0.59 &  2.88024e--04  & 1.11 &  7.78711e--03 &  1.08\\
	4.06348e--04 & 1.15 &  2.33942e--02 & 0.57 &  1.34837e--04  & 1.09 &  3.67251e--03 &  1.08\\
	\hline
	\end{tabular}}
	\caption{Errors and the convergence rates of the standard Galerkin approximation in standard and weighted norms ($\alpha=1-\lambda_1$) in the case of the L-shape domain (see Fig.~\ref{F:domains} third from left) using the (higher-order) Taylor--Hood element, computed for an exact solution $(\su_1 + \su_2, s^p_1 + s^p_2 - |\Omega|^{-1} \langle s^p_1 + s^p_2, 1 \rangle_\Omega)$. Estimated convergence rates of the errors in standard norms are for the velocity: $2\lambda_1 = 1.09$ and the pressure: $\lambda_1=0.54$.}
	\label{T:polTH}
\end{table}
\begin{figure}[ht!]
	\vspace*{-0.7cm}
	\centering
	\includegraphics[width=0.4\textwidth]{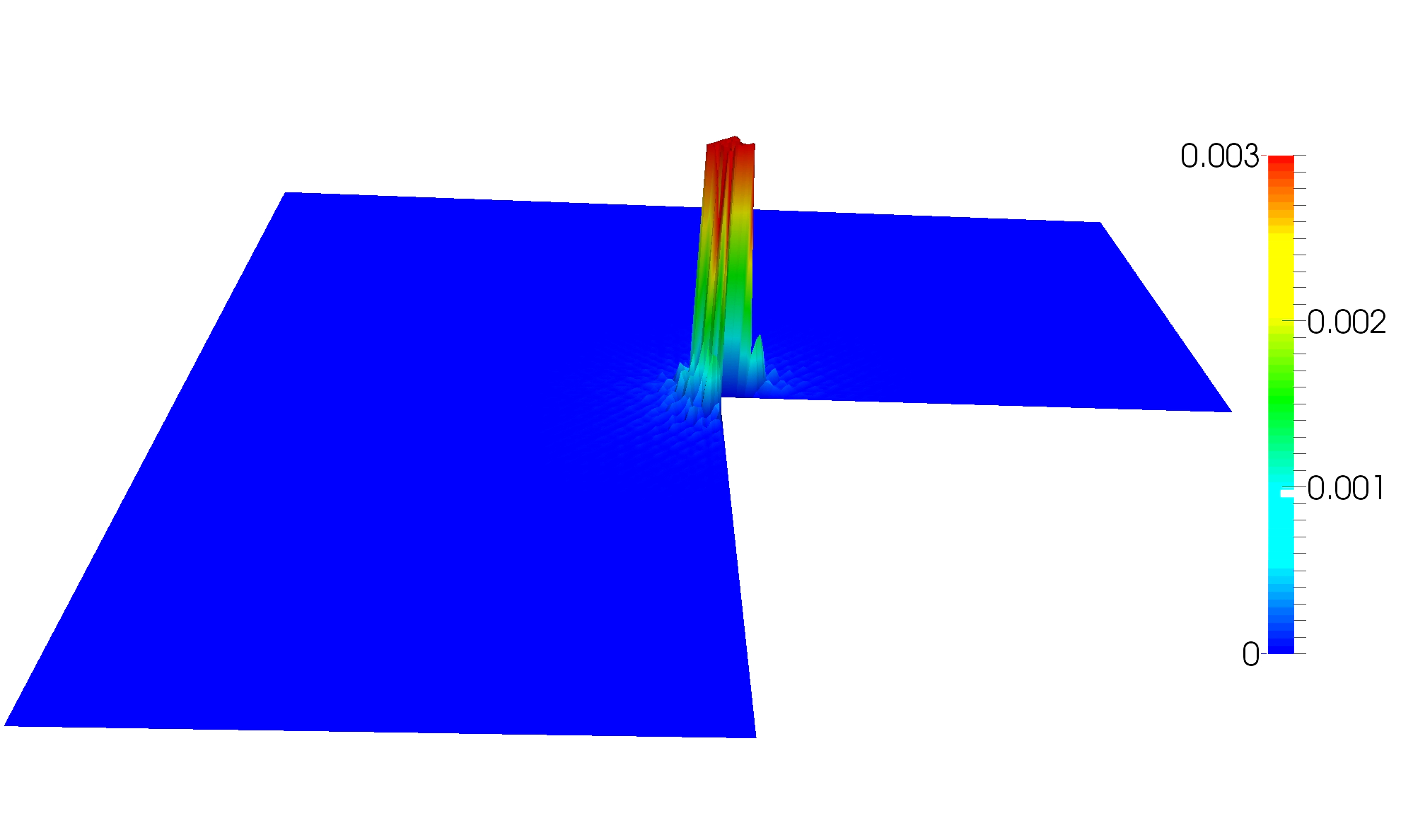}
	\hspace{0.05\textwidth}
	\includegraphics[width=0.4\textwidth]{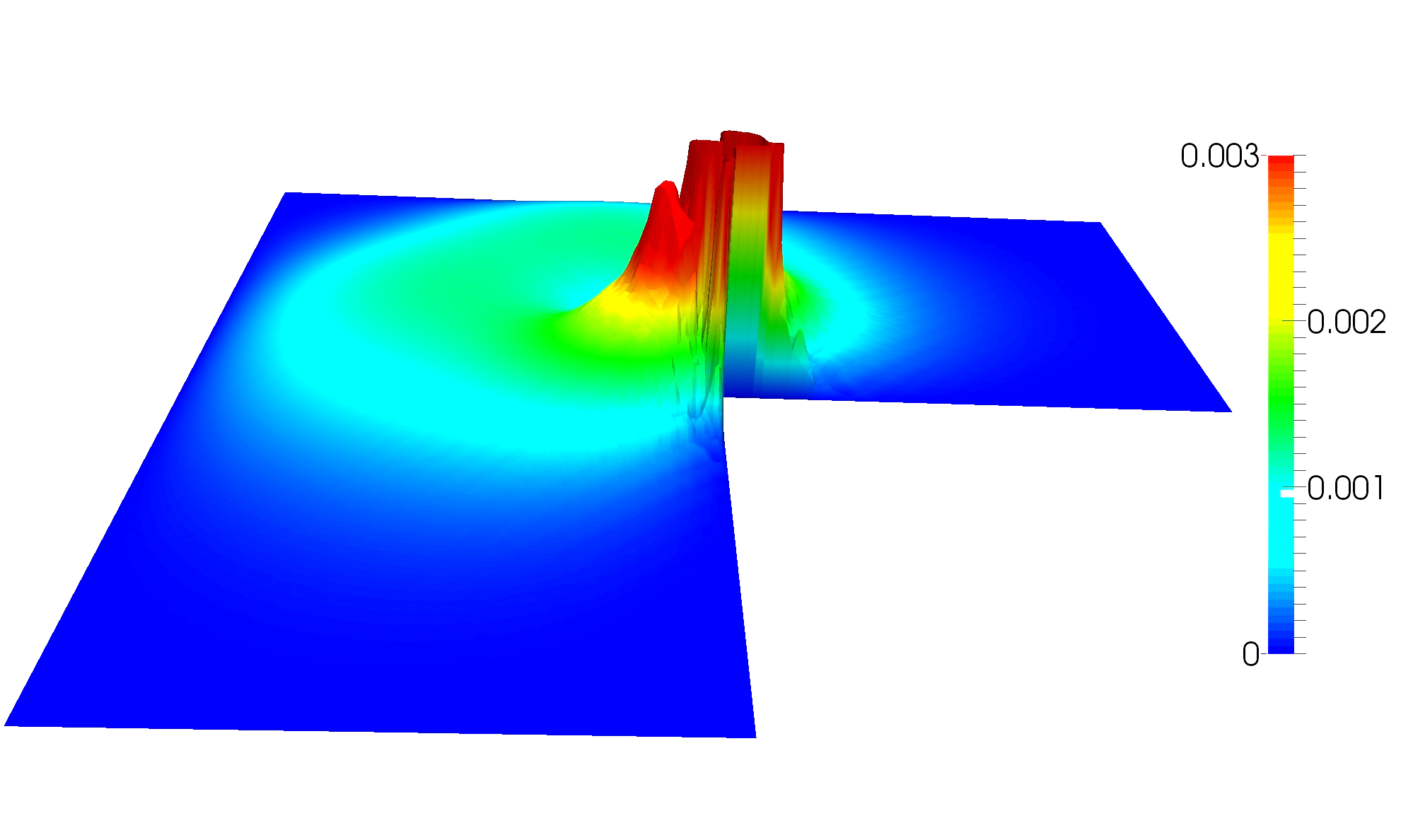}
	\caption{Illustration of the global propagation of the pollution in a L-shape domain with exact solution: $(\su_1, s^p_1 - |\Omega|^{-1} \langle s^p_1, 1 \rangle_\Omega)$ for the Taylor--Hood element; left: non-polluted  $|\su_1-I_h^2 \su_1|$, right: $|\su_1 - \su_{1,h}|$ on level $\ell = 4$.}
	\label{F:polTH}
\end{figure}

\subsection{The energy-corrected scheme}
Let us now focus on the a~priori energy-modification for the stabilized scheme. We closely follow the theory 
given in \cite{John2015}, thus it serves as a reference for additional information or rigorous proofs for the presented statements. In view of this approach, we define for a given correction vector $\gam=(\gamma_1,\gamma_2)^\top$, at this point not closely specified, a mesh-dependent bilinear form $a_h(\cdot,\cdot)$ for $\vh,\wh\in \Vh$ by:
\begin{equation}\label{defah}
	\begin{aligned}
		a_h(\wh,\vh)&\coloneqq  a(\wh,\vh) - d_h(\wh,\vh)\\
		&=\sum_{T\in \mathcal{T}_h\setminus(\mathcal{L}^1_h\cup\mathcal{L}^2_h)}a_T(\wh,\vh) + \sum_{i=1,2}(1-\gamma_i) \sum_{T\in\mathcal{L}^i_h}a_T(\wh,\vh),
	\end{aligned}
\end{equation}
where $a_T(\wh,\vh)\coloneqq\int_T\nabla \wh:\nabla\vh\,dx$,  $\mathcal{L}_h^i$ are defined as sets of elements which are in the $i-$th element layer around the singular point, and a stabilization term 
\begin{equation}\label{defch}
	\begin{aligned}
		c_h(q_h,z_h)&\coloneqq \sum_{T\in \mathcal{T}_h\setminus(\mathcal{L}^1_h\cup\mathcal{L}^2_h)} c_T(q_h,z_h) + \sum_{i=1,2}\frac{1}{1-\gamma_i}\sum_{T\in\mathcal{L}^i_h}c_T(q_h,z_h).
	\end{aligned}
\end{equation}
The stabilized energy-corrected finite element approximation to \eqref{vastab} is then formulated for the bilinear forms $a_h(\cdot,\cdot)$ and $c_h(\cdot,\cdot)$. Namely, we find $(\uhm, \phm) \in\Vh \times \Qh$ such that
\begin{equation}\label{modvastab}
	\begin{alignedat}{4}
		&a_h(\uhm,\vh) &&+ b(\vh,\phm) &&=\langle\fb,\vh\rangle_\Omega \qquad &&\forall \vh\in \Vh,\\
		&b(\uhm,\qh) &&- c_h(\qh, \phm) &&= \langle\fst,\qh\rangle_\Omega
		&&\forall \qh\in \Qh.
	\end{alignedat}
\end{equation}
By definition \eqref{defah}, and \eqref{defch}, the bilinear forms $a_h(\cdot,\cdot)$ and $c_h(\cdot,\cdot)$ differ from the ones in the standard stabilized Galerkin approximation only on an $\mathcal{O}(h)-$neighborhood of the singular point.

To measure the impact of the pollution in the $L^2-$approximation error, we define the energy defect function (also called pollution function) of the discrete scheme as:
\begin{align}\label{ghw}
	\hspace*{-0.4cm}g_h (\ub,p)\!\coloneqq a(\ub\!-\!\uhm,\ub\!-\!\uhm)\! +\! 2 b(\ub\!-\!\uhm,p\!-\!\phm)\! -\! d_h(\uhm,\uhm)\! -\! c_h(\phm,\phm).
\end{align}
Note, that the pollution function $g_h$, as defined in \eqref{ghw}, can be expressed in the case of a stable finite element formulation of the incompressible homogeneous Dirichlet problem by:
\begin{align}\label{ghwef}
	g_h(\ub,p) = a(\ub,\ub) - a_h(\uhm, \uhm),
\end{align}
hence the name energy defect function. Also, since the correction will be constructed in such a way that the pollution $g_h(\ub,p)$ is controlled, we speak about energy-corrected finite element schemes. 

\subsection{Theoretical estimates}
Let us present here the main result from \cite{John2015}, which was proven for stable linear approximations (more precisely for the \piso element). If explicitly mentioned, we additionally require a local symmetry in angular direction of the mesh around the re-entrant corner; this condition is denoted by (Sh). Note also, that although the main result is formulated for maximal interior angles in the range of  $\omega\in(\pi,\omega_3)$, the performance of the modification was shown numerically for all $\omega\in(\pi,2\pi)$. One can verify that 
\begin{align}
	\| \ub - \uhm\|_{0,\alpha}  = \mathcal{O}(h^2) \iff g_h(\ub,p) =\mathcal{O}(h^2),
\end{align}
and thus, the following theorem represents the sufficient condition for the optimal convergence. The necessary condition, i.e., the implication from left- to right-hand side, is rather trivial to prove and free of additional assumptions.
\begin{theorem}[Optimal order convergence]\label{thmmain}
	Let $\omega\in(\pi,\omega_3)$, $\omega\neq\omega_2$, $\fb\in L^2_{-\alpha}(\Omega)^2$ for some $\alpha$ such that $\alpha>1-\lambda_1$, let $(\Vh,\Qh)$ be a stable pair with $c_h(\cdot,\cdot)=0$, and, additionally, let (Sh) hold if $\omega\in[\omega_1,\omega_3)$. Moreover,  let $\gam\in\mathbb{R}^2$ be such that 
	\begin{align}\label{conv}
		g_h(\sui,\spi)  =\mathcal{O}(h^2), \quad \text{ for } \left\{
		\begin{array}{ll} 
		i=1 & \text{ if }\omega\in(\pi,\omega_2),\\
		i=1,2 & \text{ if }\omega \in (\omega_2, \omega_3),
		\end{array} \right.
	\end{align}
	where $\omega_1$, $\omega_2$ and $\omega_3$ are defined as in Fig~\ref{F:roots}. Then, the modified Galerkin approximation $(\uhm,\phm)\in\Vh\times\Qh$ converges optimally, i.e.,
	\begin{align}\label{optest}
		\|\ub-\uhm\|_{0,\alpha}\lesssim h^2 \|\fb\|_{0,-\alpha}, \quad \text{and} \quad
		\|\ub-\uhm\|_{1,\alpha}+\|p-\phm\|_{0,\alpha}\lesssim h \|\fb\|_{0,-\alpha}.
	\end{align}
\end{theorem}
First of all, let us describe the assumptions in Theorem~\ref{thmmain} and the proof-technique, after we shall discuss its generalization to the stabilized case. 

At the beginning, we would like to point out here, that $\gam=(\gamma_1,\gamma_2)^\top\in\mathbb{R}^2$ is constrained such  that the bilinear forms
 $a_h(\cdot,\cdot)$ and $c_h(\cdot,\cdot) $ remain uniformly elliptic and continuous, respectively. This means, since the element layers are disjoint, we actually require $\gam\in\mathcal{B}^\infty_R(0)$ ($\mathcal{B}^\infty_R(0)$ being a ball with center in the origin of radius $R<1$ in the $l^\infty$ norm). Also, if the first condition of \eqref{conv} has to be satisfied, i.e., $\omega\in(\pi,\omega_2)$, we represent the vector $\gam$ for simplicity of notation by $\gam = (\gamma_1,0)^\top.$ In that case, we speak about a single- (or one-) parameter correction modification. Such a case largely inherits the properties of the modification as derived in \cite{Egger2014} for the Laplace equation.

The proof of \eqref{optest} is based on duality arguments and strongly relies on the linearity of the system. Under the assumption, that the right-hand side is of increased regularity, the decomposition \eqref{updec} can be used on the exact and the approximative solutions (note that the eigenvectors $(\sui,\spi)$ are solutions of the Stokes equations themselves), and thus, the pollution in the system can be split into several parts, namely pollution caused by the singular functions, terms corresponding to the smooth remainder and the so-called cross-terms. By  standard duality arguments, a uniform inf-sup condition with respect to weighted norms, a priori bounds of the  smooth remainder  in negatively weighted norms, as introduced by \cite{Blum1990}, one can show that the terms with smooth remainder are of second order, and thus they do not pollute. The same holds for the cross-terms. However, the proof is   technically more involved.
More precisely, since the orthogonality of the eigenvectors is, in general, not preserved by the finite element approximations, we have to apply generalized Wahlbin arguments and use the local symmetry condition (Sh) on the mesh. Note here, that the drop of (Sh) for small angles comes from the sufficient regularity of the second singular function, and thus the symmetry arguments do not have to be used. The contribution to the pollution $g_h$ by the singular functions cannot be neglected, therefore it is enforced by the modification itself. This exactly reflects the condition \eqref{conv}. As one can see, for angles smaller as $\omega_2$ we have for the second singular function: $(\su_2,s^p_2)\in H^2_{-\varepsilon}(\Omega)^2\times H^1_{-\varepsilon}(\Omega)$, and thus its pollution is again a~priori of second order. This means, it does not have to be corrected. The correction vector $\gam$ is then reduced to a single scalar value, as noted above.

In Theorem~\ref{thmmain}, the weight $\alpha$ is bounded from below which originates in the assumption on the decomposition \eqref{updec}. Note, that also for the results discussed in \cite{John2015}, some additional assumptions on the upper bound are specified. These are mainly for technical reasons, namely restricting the number of singular functions in the decomposition and to ensure the validity of some auxiliary lemmas.

Let us now turn our attention to the extension of Theorem~\ref{thmmain} to stabilized finite elements. Our approach is motivated by the equivalence of the MINI and a stabilized $P_1$--$P_1$ element. For that, we define a space of element-wise bubble functions:
\begin{align}
	\mathbf{B}_h = \left\{ v_h \in  \mathcal{C}(\overline{\Omega}) : v_h|_T \in \text{span}\{\lambda_{1}^T,\lambda_{2}^T,\lambda_{3}^T\}\ \forall T \in \mathcal{T}_h \right\}^2,
\end{align}
where $\lambda_{i}^T$, $i=1,\dots,3$, denote the barycentric coordinates on the element $T$, see also \cite{Arnold1984}. The MINI element is then defined via the product $(\Vh \oplus \mathbf{B}_h) \times \Qh$.  The equivalence of the MINI element and the stabilized one, as defined above, (see \cite{brezzi-bristeau-franca-mallet-roge_1992} for the standard case) holds also for the case of the energy corrected scheme if $c_h(\cdot,\cdot)$ is selected appropriately. Namely, we have the following.
\begin{theorem}\label{equiv_MINI_pspg}
	Let $\fb$ be piecewise constant on each element $T \in \mathcal{T}_h$. Then the energy corrected MINI element discretization is equivalent to the energy corrected $P_1$--$P_1$ discretization with \textnormal{PSPG} stabilization \eqref{modvastab} where the stabilization parameter $\delta_T=\frac{c_{1,T}\, c_{2}}{\alpha_T}$. Further, the element-wise constants are determined by
	\begin{align}
		\langle \phi_T, 1 \rangle_T = c_{1,T}\, h_T^2, \qquad \langle \phi_T, 1 \rangle_T = c_{2}\, |T|, \qquad \langle \nabla \phi_T, \nabla \phi_T \rangle_T = \alpha_T,
	\end{align}
	where $\phi_T \in B_h|_T$ denotes the bubble function, scaled to a maximal value of one, on each element $T \in \mathcal{T}_h$, and $c_2=\frac{9}{20}$.
\end{theorem}
\begin{proof}
	Let us consider the energy corrected MINI finite element formulation, and denote the solution by $\ub_h = \ub_h^l + \ub_h^b$, where $\ub_h^l \in \Vh$ and $\ub_h^b = \sum_{T \in \mathcal{T}_h} \ub_{b,T} \phi_T \in \mathbf{B}_h$, with $\mathbf{u}_{b,T} \in \mathbb{R}^2$.  Note that the linear and bubble functions are $A$--orthogonal, i.e., $a(\vh^l, \phi_T\, \mathbf{e}_i) = 0$, $i = 1,2$, for all $\vh^l \in \Vh^1$ and $\phi_T \in B_h$. Let us consider the test function $\vh = \phi_T\, \mathbf{e}_i$, $i = 1, 2$, on $T$, then for the first equation of the variational formulation \eqref{modvastab}, we obtain that
	\begin{align}
		(1 - \gamma_T)\, a(\phi_T\, \ub_{b,T}, \phi_T\, \mathbf{e}_i) + b(\phi_T\, \mathbf{e}_i, p_h) = \langle \fb, \phi_T\, \mathbf{e}_i \rangle_T,
	\end{align}
	where $\gamma_T$ represents the correction parameter on the individual element. Note, $\gamma_T = 0$ for $T\in\mathcal{T}_h\setminus(\mathcal{L}^1_h\cup\mathcal{L}^2_h)$. Using integration by parts and the fact that $\phi_T = 0$ on $\partial T$ we have:
	\begin{align}
		(\ub_{b,T} \cdot \mathbf{e}_i) (1 - \gamma_T) \langle \nabla \phi_T, \nabla \phi_T \rangle_T = \langle \fb - \nabla p_h, \phi_T\, \mathbf{e}_i \rangle_T.
	\end{align}
	Then, applying the assumption of a piecewise constant $\fb$, we get that
	\begin{align}
		\ub_{b,T} = \frac{1}{\alpha_T (1 - \gamma_T)} (\fb - \nabla p_h)|_T\, \langle \phi_T, 1\rangle_T.
	\end{align}
	Using this, the bilinear form $b(\cdot,\cdot)$ can be rewritten as:
	\begin{align}
		b(\phi_T\, \ub_{b,T}, \qh) &= (\ub_{b,T} \cdot \nabla \qh)|_T\, \langle \phi_T, 1 \rangle_T = \frac{1}{\alpha_T (1 - \gamma_T)} (\fb - \nabla p_h) \cdot \nabla \qh|_T\, \langle \phi_T, 1\rangle_T^2\\
		&= \frac{c_{1,T}\, c_{2}}{\alpha_T (1 - \gamma_T)} h_T^2 \langle \fb - \nabla p_h, \nabla \qh \rangle_T,
	\end{align}
	and thus
	\begin{align}
		b(\ub_h^l, \qh) - \sum_{T \in \mathcal{T}_h} \frac{1}{ 1 - \gamma_T} \delta_T\, h_T^2 \langle \nabla p_h, \nabla \qh \rangle_T = -\sum_{T \in \mathcal{T}_h} \frac{1}{ 1 - \gamma_T} \delta_T\, h_T^2 \langle \fb, \nabla \qh \rangle_T,
	\end{align}
	which corresponds to the PSPG stabilization.
\end{proof}

In the case that the computational domain is the L-shape domain with a mesh as depicted by the third one from the left in Fig.~\ref{F:domains}, the stabilization parameter $\delta_{T} $ is independent of the element and equal to $1/160$.
The equivalence from Lemma~\ref{equiv_MINI_pspg} also motivates the extension to other stabilizations, like the Dohrman--Bochev approach \cite{bochev-dohrmann_2004} or the local projection stabilization \cite{ganesan-matthies-tobiska_2008}. In that case we proceed as before, as we scale the bilinear form $c_h(\cdot, \cdot)$ on each element $T \in \mathcal{L}_h^i$ by the term $(1 - \gamma_i)^{-1}$, for $i=1,2$. We note that this is the inverse scaling of the bilinear form
$a_h(\cdot,\cdot)$.

\subsection{Identification of the correction vector $\gam$}
Up to this point, we assumed that the vector $\gam$, chosen such that $g_h(\ub,p)=\mathcal{O}(h^2)$, is known and exists. If this is the case, then the practical realization into  an existing code is straightforward, since it only requires a simple re-scaling of a few local stiffness matrices.

The numerical determination of the unknown $\gam$ was first proposed for the Laplace equation in \cite{Egger2014} and theoretically studied in \cite{Ruede2014}. It is constructed as a limit-vector in a component-wise meaning (or a limit-value), of a sequence $\{\gam_h\}_{h>0}$ solving 
\begin{align}\label{gamhcond}
	g_h(\su_i,s^p_i) = \mathbf{0},
\end{align} 
for $i=1$ if $\omega\in(\pi,\omega_2)$, and $i=1,2$ if $\omega\in(\omega_2,2\pi)$. 
We understand the notation above in a following functional-setting sense: For each $\gam\in\mathcal{B}_R^\infty(0)$, the approximation operators $R^u:H^1_0(\Omega)^2\rightarrow \Vh$ and $R^p:L^2_0(\Omega)\rightarrow \Qh$ are considered as functions of $\gam$. Thus, on each computational level, corresponding to a certain $h>0$, we seek a root $\gam_h$ of
\begin{align}
	a(\su_i,\su_i) - \left[ a_h(R^u(\gam)\su_i, R^u(\gam)\su_i) + c_h(R^p(\gam)s^p_i, R^p(\gam)s^p_i) \right],
\end{align}
again, $i$ from the same set as above. The existence and uniqueness of the sequence $\{\gam_h\}_{h>0}$ can be proved with the help of a fixed point theorem.
Scaling arguments and the properties of the energy correction function yield the convergence of $\gam_h$, and 
we denote the asymptotic value by star, i.e., $\gam^\star = (\gamma_1^\star,\gamma_2^\star)^\top = \left(\lim_{h\rightarrow0^+}\gamma_{1,h},\lim_{h\rightarrow0^+}\gamma_{2,h}\right)^{\!\!\top}$. 
Then one can show that for $\gam=\gam^\star$ the assumption \eqref{conv} is satisfied. To do so one has to construct a suitable subsequence and exploit the properties of the Jacobi matrix of $g_h$. The computation of $\gam^\star $ can be
accelerated by a combination of a nested hierarchical Newton scheme with some
extrapolation technique on a coarse element patch.

 We remark that $\gam^\star$ depends only on the interior angle $\omega$, on the local element patch and the finite element type, but not on the mesh-size nor on  the  rest of the mesh/domain. Moreover, many re-entrant corners can be handled by local modifications each of which carrying its own correction parameter.

\subsection{Numerical examples}
In this section, we consider several numerical examples for the stabilized schemes \eqref{vastab} and \eqref{modvastab}. The correction parameters are estimated as described above, see Tab.~\ref{T:ompar}, for the initial grids as depicted in Fig.~\ref{F:domains}, i.e., for the scenarios where $\omega\in\{8/7\pi,5/4\pi,3/2\pi,7/4\pi\}$. We always consider a computational setting for which $(\su_1 + \su_2, s^p_1 + s^p_2 - |\Omega|^{-1} \langle s^p_1 + s^p_2, 1 \rangle_\Omega)$ is the exact solution. Note, that for $\omega=8/7\pi$ we do not require the local symmetry of the mesh at the singular point, for both $\omega=8/7\pi, 5/4\pi$ a single-parameter correction can be used, while $\omega=3/2\pi, 7/4\pi$ cases require two correction parameters. The last case $\omega=7/4\pi$ represents the extension of the derived theory for angles bigger than $\omega_3$. For simplicity, we represent only the results for the velocity, see Tab.\ref{T:omconv}, more precisely, we show the weighted $L^2_{\alpha}(\Omega)$ error convergences and their rates. 
\begin{figure}[ht!]
	\centering
 	\includegraphics[width=0.75\textwidth]{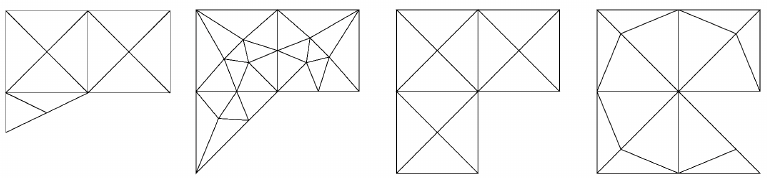}
	\caption{Initial grids for several types of the re-entrant corner, namely $\omega\in\{8/7\pi,5/4\pi,3/2\pi,7/4\pi\}$. All the geometries are bounded by a unit square, particularly $\Omega_{3/2\pi} = (-1,1)^2 \backslash ([0, 1] \times [-1, 0])$.}
	\label{F:domains}
\end{figure}
\begin{table}[ht!]
	\setlength\tabcolsep{5pt}
	\centering
	\small{\begin{tabular}{| c | c  || c | c | }
	\hline
	$\omega$ & $\gam^\star$ (single correction) & $\omega$ & $\gam^\star$ (two corrections)\\ \hline
	\hline &&& \\[-1.0em]
	$8/7 \pi$ & $(0.011364, 0.0 )^\top$ & $3/2 \pi$ & $(\phantom{-}0.161708, -0.183984)^\top$ \\ 
	$5/4 \pi$ & $(0.017676, 0.0 )^\top$ & $7/4 \pi$ & $(-0.483852, \phantom{-}0.273889)^\top$ \\
	\hline
	\end{tabular}}
	\caption{Estimated asymptotic $\gam^\star$ for the geometries and local element patches  of Fig.~\ref{F:domains}.}
	\label{T:ompar}
\end{table}
\begin{table}[ht!]
	\setlength\tabcolsep{3pt}
	\centering
	\small{\begin{tabular}{| c | c || c | c || c | c || c | c | }
	\hline
	\multicolumn{2}{|c||}{$\omega = 8/7\pi$}  & \multicolumn{2}{c||}{$\omega = 5/4\pi$} & \multicolumn{2}{c||}{$\omega = 3/2\pi$} & \multicolumn{2}{c|}{$\omega = 7/4\pi$}   \\ \hline
	\hline &&&&&&& \\[-1.0em]
	$\| \ub - \uhm \|_{0,\alpha}$ & eoc & $\| \ub - \uhm \|_{0,\alpha}$ & eoc & $\| \ub - \uhm \|_{0,\alpha}$ & eoc & $\| \ub - \uhm \|_{0,\alpha}$ & eoc  \\[0.5ex] \hline
	6.80763e--02 &   --    &  6.43853e--02  & --    & 9.87423e--02 & --       & 6.31650e--02 & --          \\
	2.21098e--02 & 1.62 &  2.36951e--02 & 1.44 & 3.76712e--02 & 1.39  & 1.79778e--02 & 1.81 \\
	6.06479e--03 & 1.87 &  6.78762e--03 & 1.80 & 1.28515e--02 & 1.55  & 4.55198e--03 & 1.98 \\
	1.59057e--03 & 1.93 &  1.81517e--03 & 1.90 & 3.48892e--03 & 1.88  & 1.06693e--03 & 2.09 \\
	4.07657e--04 & 1.96 &  4.68415e--04 & 1.95 & 8.87484e--04 & 1.97  & 2.47722e--04 & 2.11 \\
	1.03206e--04 & 1.98 &  1.18907e--04 & 1.98 & 2.19857e--04 & 2.01  & 5.82728e--05 & 2.09 \\
	2.59642e--05 & 1.99 &  2.99543e--05 & 1.99 & 5.37256e--05 & 2.03  & 1.41046e--05 & 2.05 \\
	\hline
	\end{tabular}}
	\caption{Errors and the convergence rates of the velocity approximation measured in the weighted norm with $\alpha = 1 - \lambda_1$.}
	\label{T:omconv}
\end{table}

To demonstrate the convergence of the pressure part of the solution, we additionally include Tab.\ref{T:Lconv} for $\omega=3/2\pi$, presenting also the rates and errors in the case of the standard $L^2(\Omega)-$norm. For the velocity, we obtain in that case the best approximation order, namely $\lambda_1+1 = 1.54$, instead of the polluted rate $2\lambda_1=1.09$, see \eqref{pol}. In the case of the pressure-error, we have increased the weight to $\alpha_p=\alpha+0.5$, to restore its superconvergence behavior.
\begin{table}[ht!]
	\setlength\tabcolsep{3pt}
	\centering
	\small{\begin{tabular}{| c | c | c | c || c | c | c | c | }
	\hline
	\multicolumn{4}{|c||}{standard norm}  & \multicolumn{4}{c|}{weighted norm}  \\ \hline
	\hline &&&&&&& \\[-1.0em]
	$\| \ub - \uhm \|_{0}$ & eoc & $\| p - \phm  \|_{0}$ & eoc & $\| \ub - \uhm \|_{0,\alpha}$ & eoc & $\| p-\phm \|_{0,\alpha_p}$ & eoc  \\[0.5ex] \hline
	4.70326e--02 &   -- & 7.92268e--01 &  --  & 9.87423e--02 & --    & 4.03685e--01 & --   \\
	3.87065e--02 & 2.81 & 4.61252e--01 & 0.78 & 3.76712e--02 & 1.39  & 2.04358e--01 & 0.98 \\
	1.88110e--02 & 1.04 & 2.30692e--01 & 0.99 & 1.28515e--02 & 1.55  & 5.03504e--02 & 2.02 \\
	7.07091e--03 & 1.41 & 1.25586e--01 & 0.87 & 3.48892e--03 & 1.88  & 1.83549e--02 & 1.46 \\
	2.51514e--03 & 1.49 & 6.98617e--02 & 0.84 & 8.87484e--04 & 1.97  & 6.74790e--03 & 1.44 \\
	8.78492e--04 & 1.51 & 3.97820e--02 & 0.81 & 2.19857e--04 & 2.01  & 2.37574e--03 & 1.50 \\
	3.04097e--04 & 1.53 & 2.33040e--02 & 0.77 & 5.37256e--05 & 2.03  & 8.20768e--04 & 1.53 \\
	\hline
	\end{tabular}}
	\caption{Convergence results of the velocity and pressure modified approximation measured in the standard and weighted norms, with minimal admissible value $\alpha = 1-\lambda_1$ and increased weight for pressure $\alpha_p=\alpha+0.5$ to restore the superconvergence, in the case of L-shape domain.}
	\label{T:Lconv}
\end{table}

Finally, we return to the motivational problem of Fig.~\ref{F:polTH}, i.e., a higher order approximation, and illustrate on the left  of Fig.~\ref{F:polmodTH} the reduction of the pollution by the energy-corrected scheme, as derived above, for the same computational setting. On the right, we show the solution of an optimal control problem on a domain with multiple (50) re-entrant corners. Here at each corner, the modified scheme can be applied locally with the same correction parameters as for the L-shape domain. Many quantities of interest such as stress intensity factors, eigenvalues or the flow rate can then accurately be determined on uniformly refined meshes, and no mesh grading techniques are required.
 
\begin{figure}[ht!]
	\hspace*{0.2cm}
  \begin{minipage}[c]{0.45\textwidth}
    \includegraphics[width=0.85\textwidth]{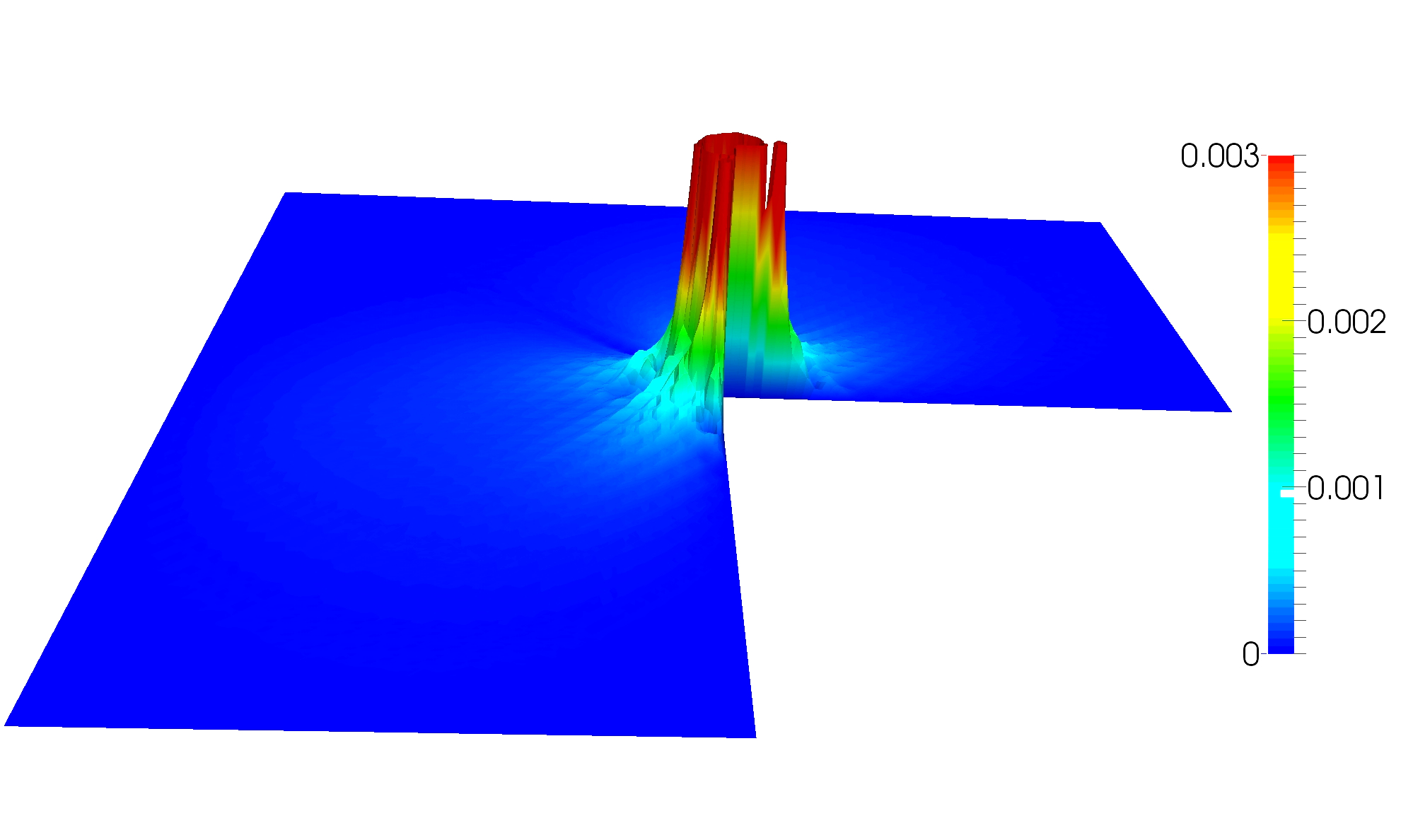}
  \end{minipage}
  \begin{minipage}[c]{0.55\textwidth}
   \includegraphics[width=\textwidth]{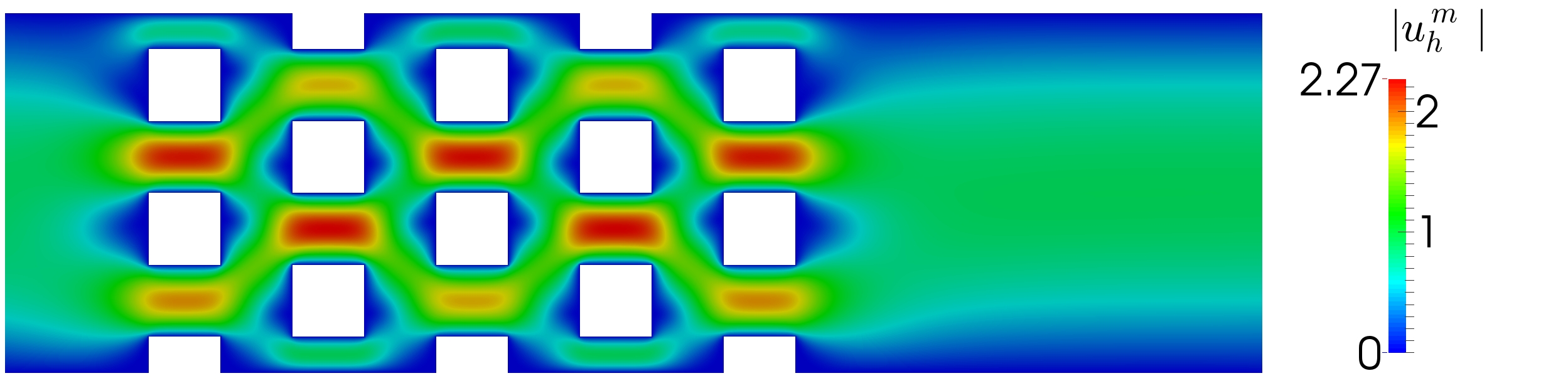}
  \end{minipage}
 \caption{Left: Reduced pollution effect for the modified Taylor--Hood scheme
(compare with Fig.~\ref{F:polTH}); Namely we plot $|\ub-\uhm|$. 
Right: Multiple re-entrant corners.}
		\label{F:polmodTH}
\end{figure}

\section{Multigrid 
and fault tolerant algorithms}\label{sec:solver}
In this section, we consider the realization of efficient multigrid solvers and fault tolerant algorithms for the Stokes system. Our implementation 
employs the hierarchical hybrid grids (HHG) finite element framework, see, e.g.,
\cite{bergen-huelsemann_2004,bergen-gradl-ruede-huelsemann_2006}. 
HHG  combines the flexibility of unstructured finite element meshes with the performance advantage of structured grids in a block-structured approach.
It shows excellent scalability up to the largest supercomputers
available \cite{gmeiner-ruede-stengel-waluga-wohlmuth_2015}
and reaches parallel textbook multigrid efficiency
\cite{gmeiner-ruede-stengle-waluga-wohlmuth_2015_2}.

In the following, we denote by $n_u = {\rm dim}\, \mathbf{V}_h$ and $n_p = {\rm dim}\, Q_h$ the number of degrees of freedom for velocity and pressure, respectively. Then, the 
isomorphisms $\mathbf{u}_h \leftrightarrow \mathbf{u} \in \mathbb{R}^{n_u}$ and $p_h \leftrightarrow \mathbf{p} \in \mathbb{R}^{n_p}$ are 
obvious.
The discrete variational formulation \eqref{vastab} is then equivalent to the linear system
\begin{align}\label{eq:linear-system}
\mathcal{K}
\begin{pmatrix}
\mathbf{u} \\
\mathbf{p}
\end{pmatrix}
:=
\begin{pmatrix}
A & B^\top \\
B & -C
\end{pmatrix}
\begin{pmatrix}
\mathbf{u} \\
\mathbf{p}
\end{pmatrix} 
=
\begin{pmatrix}
\mathbf{f}\phantom{s} \\
\fstb
\end{pmatrix},
\end{align}
where $\mathcal{K} \in \mathbb{R}^{(n_u + n_p) \times (n_u + n_p)}$.
The matrix A 
has a $3\times3$ diagonal block structure 
consisting of three discrete representations of the Laplace operator. 

In the following, we describe a multigrid solver for this structure and 
a multigrid based fault tolerant solution strategy. 

\subsection{Multigrid solver}\label{subsec:multigrid}
Iterative methods for solving the saddle point system 
can e.g.~be constructed 
based on the pressure Schur complement (SCG) or on a 
preconditioned MINRES method,
see, e.g., \cite{elman-silvester-wathen_2005, verfuerth_1984}.

Here, we consider
an {\em all-at-once} multigrid method for the linear system \eqref{eq:linear-system}. 
The key point in a multigrid method for the entire saddle point problem is the construction of a suitable smoother.
We shall consider Uzawa type smoothers, see, e.g., \cite{bank-welfert-yserentant_1990, schoeberl-zulehner_2003, zulehner_2002}. 
This idea is based on a preconditioned Richardson method, where in 
iteration $k+1$ the 
system for $(\mathbf{u}_{k+1}, \mathbf{p}_{k+1})$ 
must be solved
\begin{align}\label{precond_richardson}
\hat{\mathcal{K}}
\begin{pmatrix}
\mathbf{u}_{k+1} - \mathbf{u}_k\\
\mathbf{p}_{k+1} - \mathbf{p}_k
\end{pmatrix}
=
\begin{pmatrix}
\mathbf{f}\phantom{s} \\
\fstb
\end{pmatrix}
- \mathcal{K}
\begin{pmatrix}
\mathbf{u}_k\\
\mathbf{p}_k
\end{pmatrix}.
\end{align}
Here $\mathcal{K}$ denotes the system matrix \eqref{eq:linear-system} of the Stokes equations and $\hat{\mathcal{K}}$ an appropriate preconditioner. The corresponding iteration matrix is then given by $\mathcal{M} \coloneqq I - \hat{\mathcal{K}}^{-1} \mathcal{K}$. We  note that for the convergence of \eqref{precond_richardson}, the analysis of $\mathcal{M}^k$ is needed, 
while for showing a qualitative 
smoothing property, the matrix polynomial $(I - \mathcal{M})\mathcal{M}^k$ 
must be analyzed \cite{john-wohlmuth-zulehner, schoeberl-zulehner_2003}. 
Approaches towards a full quantitative multigrid analysis
based on Fourier techniques are presented in \cite{Gaspar_2014}.
In general several choices for $\hat{\mathcal{K}}$ are available, we restrict ourselves to
\begin{align*}
\hat{\mathcal{K}} \coloneqq
\begin{pmatrix}
\hat{A} & 0 \\
B & -\hat{S}
\end{pmatrix},
\end{align*}
where $\hat{A}$ and $\hat{S}$ are suitable smoothers for the upper left block $A$ and the pressure Schur complement $S \coloneqq B A^{-1} B^\top + C$, respectively. The application of $\hat{\mathcal{K}}$ results in the algorithm of the inexact Uzawa method, where we smooth the velocity part in a first step and in a second step the pressure, i.e.
\begin{align*}
&1.\quad \mathbf{u}_{k+1} = \mathbf{u}_k + \hat{A}^{-1} (\mathbf{f} - A \mathbf{u}_k - B^\top \mathbf{p}_k),\\
&2.\quad \mathbf{p}_{k+1} = \mathbf{p}_k + \hat{S}^{-1} (B \mathbf{u}_{k+1} - C \mathbf{p}_k - \fstb).
\end{align*}
Other variants of Uzawa-type multigrid methods, such as the symmetric version, are for instance considered in \cite{schoeberl-zulehner_2003}.
In each smoothing step, we consider for the velocity smoother $\hat{A}$, a combination of a forward and backward communication reducing  variant of a Gau\ss--Seidel 
scheme with an additional row-wise red-black coloring in each of the parallel distributed  subdomains. 
Within the parallel data structures this results 
in a combined block Jacobi- and symmetric Gau\ss--Seidel updating scheme.
For the pressure, we use a SOR applied on the matrix $C$
with the under-relaxation parameter 
$\omega = 0.3$.
These smoothers are then applied within a multigrid V-cycle, where we
vary the number of smoothing steps 
from 3 pre- and 3 post smoothing steps on the finest level to 5 pre- and 5 post-smoothing steps on
coarser levels. 
This leads to a 
good efficiency.
%

As a solver on the coarsest level, several 
choices are available.
This includes parallel sparse direct solvers, such as PARDISO \cite{pardiso2}, or algebraic multigrid methods, such as hypre \cite{falgout-jones-yang_2006}.
For this 
performance study, we 
avoid the dependency
on external libraries and 
employ 
a preconditioned MINRES as 
coarse-grid solver
for the saddle point system \eqref{eq:linear-system}. 
The preconditioner consists of CG and a lumped mass matrix $M$.
The 
relative accuracy 
is set to $\varepsilon = 5 \cdot 10^{-3}$. 

\subsubsection*{Numerical examples}
In the following, we present several numerical examples, illustrating the robustness and scalability of our solver. As a computational domain, we consider the unit cube $\Omega = (0,1)^3$ with an initial mesh $\mathcal{T}_{-2}$ consisting of 6 tetrahedra. 
The coarse grid $\mathcal{T}_0$ is a 
twice refined initial grid, where each tetrahedron is decomposed into 8 new ones.
The right-hand side $\mathbf{f} = \mathbf{0}$ and for the initial guess $\mathbf{x}_0 = (\mathbf{u}_0, \mathbf{p}_0)^\top$, we distinguish between velocity and pressure part.
Here $\mathbf{u}_0$ is a 
random vector with values in $[0,1]$, while the 
initial random pressure 
$\mathbf{p}_0$ is scaled by $h^{-1}$, i.e., ~with values in $[0, h^{-1}]$. 
This is important in order to 
represent a less regular pressure, since the velocity is generally considered 
as a $H^1(\Omega)$ and the pressure as a $L^2(\Omega)$ function. 
Unless 
further specified, we 
use the relative reduction  of 
the residual by $\varepsilon = 10^{-8}$ as a stopping criteria. 
Numerical examples are performed on a low cost desktop machine, 
an Intel Xeon CPU E2-1226 v3, 3.30GHz and 32 GB shared memory
for serial computations 
and on JUQUEEN (J\"ulich Supercomputing Center, Germany)
for massively parallel computations.

\definecolor{darkgreen}{rgb}{0.125,0.5,0.169}
\begin{table}[h]
{\small
\setlength\tabcolsep{4pt}
\centering
\raisebox{.1\height}{\parbox{0.55\textwidth}{
\begin{tabular}{| c | c || c | r || c | r |}
\hline \multicolumn{2}{|c||}{}  & \multicolumn{2}{c||}{} & \multicolumn{2}{c|}{} \\[-1.em]
\multicolumn{2}{|c||}{}  & \multicolumn{2}{c||}{SCG} & \multicolumn{2}{c|}{Uzawa}\\
\hline &&&&& \\[-1.0em]
$L$ & DoF & iter & \multicolumn{1}{c||}{TtS [s]} & iter & \multicolumn{1}{c|}{TtS [s]} \\
\hline \hline
& & & & & \\[-2.5ex]
2 &  $1.4 \cdot 10^{4\, }$	& 26	&     0.11  &  9 &   0.08\\ 
3 &  $1.2 \cdot 10^{5\, }$	& 28	&     0.56  &  8 &   0.29\\ 
4 &  $1.0 \cdot 10^{6\, }$	& 28	&     3.33  &  8 &   1.79\\ 
5 &  $8.2 \cdot 10^{6\, }$	& 28	&    24.28  &  8 &  12.70\\
6 &  $6.6 \cdot 10^{7\, }$	& 31	&   205.84  &  8 &  95.85\\
7 &  $5.3 \cdot 10^{8\, }$	& \multicolumn{2}{c||}{out of mem.} &  8 & 730.77\\
\hline
\end{tabular} 
}}
\begin{minipage}{0.39\textwidth}%
\centering
\begin{tikzpicture}
  \begin{axis}[
    ybar,
    enlarge x limits=0.55,
		y label style={at={(axis description cs:-0.14,.5)},anchor=south},
    ylabel={\texttt{\#} operator evaluations},
    symbolic x coords={SCG, Uzawa},
    xtick=data,
    width = 5.5cm,
    height = 5.cm,
    ]
    \addplot[color = blue, fill = blue, fill opacity=0.5] coordinates {(SCG,464) (Uzawa,130)};
    \addplot[color = red, fill = red, fill opacity=0.5] coordinates {(SCG,72) (Uzawa,140)};
    \addplot[color = darkgreen, fill = darkgreen, fill opacity=0.5] coordinates {(SCG,36) (Uzawa,70)};
    \addplot[color = gray, fill = gray, fill opacity=0.5] coordinates {(SCG,32)};
    \legend{$A$, $B$, $C$, $M$}
  \end{axis}
\end{tikzpicture}
\end{minipage}
\caption{Iteration numbers and time-to-solution (TtS) (left) and number of operator evaluations for $L = 6$ (right) for the serial case.}
\label{T:ser}
}
\end{table}
In Tab.~\ref{T:ser}, we present results for several refinement levels $L$ and a fixed coarse grid
$\mathcal{T}_0$. 
We observe that the 
multigrid 
convergence rates are independent of the mesh size, as expected.
We also include results for the Schur complement CG method
\cite{gmeiner-ruede-stengel-waluga-wohlmuth_2015}, 
where we observe that the Uzawa multigrid solver is 
typically a factor of two faster and is more memory efficient.
The number of operator evaluations for the individual blocks, which are computed by $n_\text{op} = \sum_{\ell=0}^L 8^{\ell - L} n_{\text{op}, \ell}$, where $n_{\text{op}, \ell} \in \mathbb{N}$ denotes the number of operator evaluations on 
level $\ell$, are presented too. 
These numbers are also reflected by the time to solution of the individual solvers.\\

\begin{table}[ht]
\centering
{\small
\begin{tabular}{|r r r || l | r | c | c |}
        \hline
	\multicolumn{1}{|c}{Nodes} 	& \multicolumn{1}{c}{Threads}	& \multicolumn{1}{c||}{DoF}	& iter & TtS [s] & T$_{\!w}$ [s] & T$_{\!c}$[\%]\\ \hline \hline
	& & & & & &\\[-2.5ex]
	1		& 24			& $6.6 \cdot 10^{7\ }$	& 7 (1) 	&  65.71 \ 	&  58.80 &	10.5\\ 
	6		& 192		        & $5.3 \cdot 10^{8\ }$	& 7 (3) 	&  80.42 \ 	&  64.40 &	19.9\\
	48		& 1\,536		& $4.3 \cdot 10^{9\ }$	& 7 (5) 	&  85.28 \ 	&  65.03 &	23.7\\
	384		& 12\,288		& $3.4 \cdot 10^{10}$	& 7 (10)	&  93.19 \ 	&  64.96 &	30.3\\ 
    3\,072		& 98\,304		& $2.7 \cdot 10^{11}$	& 7 (20)	& 114.36 \ 	&  66.08 &	42.2\\
  24\,576		& 786\,432	        & $2.2 \cdot 10^{12}$	& 8 (40)	& 177.69 \ 	&  78.24 &	56.0\\
	\hline
\end{tabular}
\medskip
\caption{Weak scaling result times for the Uzawa multigrid method; number of iterations, including number of iterations of the preconditioner on the coarse grid, time-to-solution (TtS), TtS without the time spent by the coarse grid iterations (T$_{\!w}$), and the time spent by the coarse grid in percents (T$_{\!c}$).}
\label{T:itertime_lap_par}
}
\end{table}
Scaling results are present for up to $786\, 432$ threads and more than 
$10^{12}$ degrees of freedom (DoF). 
We use 
32 threads per node and 6 tetrahedra 
in the initial mesh per thread. 
The finest computational level is obtained by a 7 times refined initial mesh.
Excellent scalability for the multilevel part of the algorithm is achieved, see Tab. \ref{T:itertime_lap_par}. 
We point out that 
the coarse grid solver in these examples 
considered as a stand-alone solver lacks optimal complexity. 
However, even for the largest 
example with more than $10^{12}$ DoF it still requires less than 60\% of the overall time. 


\smallskip
\subsection{Fault tolerant algorithms}
In the future  
era of exa-scale computing systems, highly scalable implementations will execute up to
billions of parallel threads on millions of compute nodes. In this scenario,
fault tolerance will become a necessary property of hardware,
software and algorithms.
Nevertheless, nowadays commonly used redundancy approaches,
e.g., check-pointing, will be too costly, due to the memory and energy consumption. Our approach, on the other hand, incorporates the resilience strategy directly into the 
multigrid solver, and thus 
does not require additional memory.

In \cite{huber-gmeiner-ruede-wohlmuth_2015}, we 
introduce a methodology and data-structure to efficiently recover lost data due to a processor crash (hard fault) when solving elliptic PDEs with 
multigrid algorithms.
We consider 
a fault model 
where a processor stores the mesh data of a subdomain including all its refined levels
in the multigrid hierarchy.
Therefore, when a processor crashes 
all data is lost in the faulty domain $\Omega_F \subset \Omega$.
The healthy domain $\Omega_H \subset \Omega$ is unaffected by the fault, and 
data in this domain remains available.
Further, we denote by 
$\Gamma \coloneqq \partial\Omega_F \cap \partial \Omega_H$ the interface. 
The nodes associated with $\Gamma$ 
are used to communicate between neighboring processors by introducing {\em ghost copies}
that 
redundantly exist on different processors. 
Therefore, a complete recovery of these nodes is 
possible without additional redundant storage.

In order to recover the lost nodal values $(\mathbf{u}_F, \mathbf{p}_F)$ in $\Omega_F$,
we propose to solve a local subproblem in $ \Omega_F$ of \eqref{stokes} with
Dirichlet boundary conditions 
on $\Gamma$ for velocity and pressure, respectively.
To guarantee that the local system is uniformly well-posed, we formally 
include a compatibility condition obtained from the normal components of the velocity.
If the local  solution in $ \Omega_F$ is computed while the global process is halted,
then this is a local recovery strategy. 
For a first analysis, we assume that the local recovery 
is free of cost, i.e., that it can be executed in zero time. 

In \cite{huber-gmeiner-ruede-wohlmuth_2015} this local strategy is extended to become 
a global recovery strategy.
To this end, the solution algorithm proceeds asynchronously in the faulty and
the healthy domain such that no process remains idle.
Temporarily, the two subdomains are decoupled at the interface $\Gamma$,
and the recovery process is 
accelerated by delegating more compute resources to it.
This acceleration is termed the {\em superman strategy}.
%
Once the recovery has proceeded far enough and has caught up with the regular solution process,
both subdomains are re-coupled and the regular global iteration is resumed.
These approaches result in a time- and energy-efficient recovery of the lost data.

\smallskip
\subsubsection*{Numerical examples}
%
\begin{figure}
\begin{minipage}[h]{0.5\textwidth}
\includegraphics[width=0.9\textwidth]{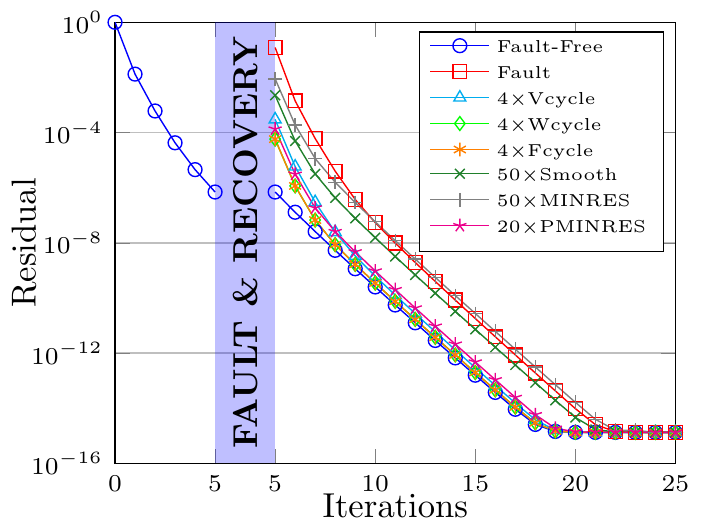}
\end{minipage}
\begin{minipage}[h]{0.5\textwidth}
   \begin{tabular}{r|cccc}
   \hline
	  &  \multicolumn{4}{c}{DN Strategy $\eta_{\text{super}} =4$}  \\
   Size  &   ${n_H}=0$&$2$&$3$&$4$\\
	  \hline \hline
$769^3$	&	19.21		&	0.05	&	-0.38	&	4.22	\\

$1\,281^3$&	21.27		&	-0.15	&	-0.68	&	3.95	\\

$2\,305^3$&	18.50		&	-0.33 & -0.87	&	3.76	\\

$4\,353^3$&	19.74		&	2.58	&	4.61	&	9.24\\
\hline
\end{tabular}
\end{minipage}
\caption{Left: Convergence of the relative residual for different local
recovery strategies, Right: Global recovery in terms of a domain
partitioning concept with over-balancing by a factor of
$\eta_{\text{super}} =4$ in the case of the Laplace operator.}
\label{fig:fault}
\end{figure}
In Fig. \ref{fig:fault} (left), we consider a test scenario in $\Omega=(0,1)^3$ in
which we continuously apply multigrid V(3,3)-cycles of
the Uzawa multigrid method introduced in Sect.~\ref{subsec:multigrid}.
In total 23 iterations are needed to reach the round-off limit of $10^{-15}$. 
During the iteration, a fault is provoked after 5 iterations.
This affects 2.1\% of the pressure and velocity unknowns.
We continue with global solves after the recovery by solving  the local auxiliary problem iteratively
in $\Omega_F$. 
As approximate subdomain solvers we compare the fine grid Uzawa-smoother, 
the minimal residual method, the block-diagonal preconditioned \mbox{MINRES} (PMINRES) method
and V, F, W-cycles in the 
variable smoothing variant from Sect.~\ref{subsec:multigrid}. 
For the block-preconditioner in case of PMINRES a standard V-cycle
and a lumped mass-matrix $M$ are used.
Fig.~\ref{fig:fault} (left)
displays also the cases in which no fault appears (fault-free) and when no recovery
is performed after the fault.

After the fault, we observe that the residual jumps up
and  when no recovery is performed, the iteration 
must start almost from the beginning.
A higher pre-asymptotic convergence rate after the fault 
helps to catch up, so that only four additional iterations are required.
This delay can be further reduced by a local recovery computation,
but  only local multigrid cycles are found to be efficient recovery methods.
Four local recovery multigrid cycles are sufficient
such that the round-off limit is reached with the same number of iterations
as in the fault-free case. 
The other recovery methods either yield no significant improvement 
or too many recovery iterations would be required. 

The table on the right of Fig.~\ref{fig:fault} summarizes the performance of the global recovery
in terms of the time delay (in seconds compute time) 
as compared to an iteration without faults.
The tests are performed for a large Laplace problem discretized with
up to almost $10^{11}$ unknowns.
The undisturbed solution  takes 50.49 seconds with
14\;743 cores on JUQUEEN.
Two faults are provoked, one after 5 V-cycles, one after 9.
Both faults are treated with both the
global recovery strategy and a local \emph{superman} process that is 
$\eta_{\text{super}} =4$ times as fast a a regular processor. 
Faulty and healthy domain remain decoupled for ${n_H}$ cycles
with the Dirichlet-Neumann strategy, i.e., solving a Dirichlet  problem on $\Omega_F$ and a Neumann problem on $\Omega_H$,
then the regular iteration is resumed.
The case ${n_H}=0$ corresponds to performing no recovery at all and leads to a
delay of about 20 seconds. 
By the superman recovery and global re-coupling
after ${n_H}=2$ cycles, the delay can be reduced to just a few seconds.
In some cases the fault-affected computation is even faster
than the regular one, as indicated by negative
time delays in the table.


\section{A posteriori mass correction}\label{sec:mc}
%
%
A common problem of continuous pressure finite element methods, such as the stabilized scheme considered in this work, is that they do not preserve the physical concept of mass-conservation in a local sense. Although the discrete solution $\mathbf{u}_h \in \mathbf{V}_h$ satisfies a weak form of the mass-balance equation, it can, in general, not be considered as element-wise mass conservative. Due to the lack of element-piecewise constants in the pressure space, we cannot expect a property of the form 
$$
\int_{\partial T} \mathbf{u}_h\cdot\mathbf{n} \,{\rm d}s = 0, \quad T \in \mathcal{T}_h,
$$
to hold, which is a crucial prerequisite to enable an element-by-element post-processing of the discrete solution to obtain strongly divergence-free velocities. However, we will demonstrate  how a local conservation property can be obtained also for continuous pressure interpolations. For this we need to change the viewpoint which requires some additional notation that is introduced next.

\subsection{Dual mesh}
Based on the triangulation $\mathcal{T}_h$, which we  refer to $\mathcal{T}_h$ as the \emph{primal} mesh, we construct a \emph{dual} mesh as it is often done for finite-volume methods. To be precise, let $\mathcal{T}_i \subseteq \mathcal{T}_h$ denote the nodal patch associated with the vertex $\mathbf{x}_i$, $i=1,\dots,N$, of the primal mesh, and let $\lambda^T_i(\mathbf{x})$ denote the barycentric coordinate of $\mathbf{x} \in \mathbb{R}^d$ with respect to the vertex $\mathbf{x}_i$ belonging to any element $T \in \mathcal{T}_i$. We define
$$
B_i^T := \{ \mathbf{x}\, :\, \lambda_i^T(\mathbf{x}) \ge \lambda_j^T(\mathbf{x}),~i \ne j \},
$$
and we set
$
B_i := \bigcup_{T \in \mathcal{T}_i} B_i^T.
$
This construction results in a partition of $\Omega$ into $d$-polytopes $B_i$ which we shall refer to as the set $\mathcal{B}_h := \{ B_i, i=1,\dots,N \}$ of \emph{control volumes}. For an illustration of typical control volumes in $d=2,3$ dimensions, we refer to Fig.~\ref{fig:dualcells}.
\begin{figure}[hpbt]
\centering
    	\includegraphics[width=.3\textwidth]{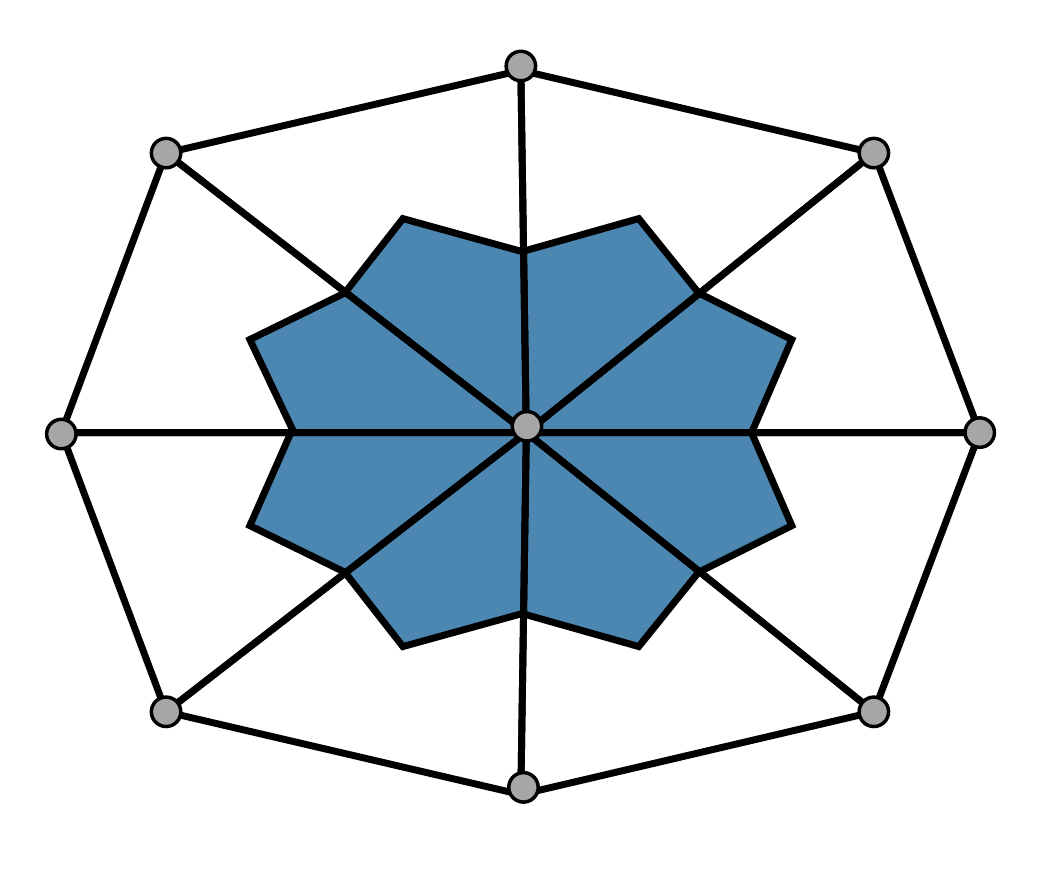}
\hspace{3em}
    	\includegraphics[width=.35\textwidth]{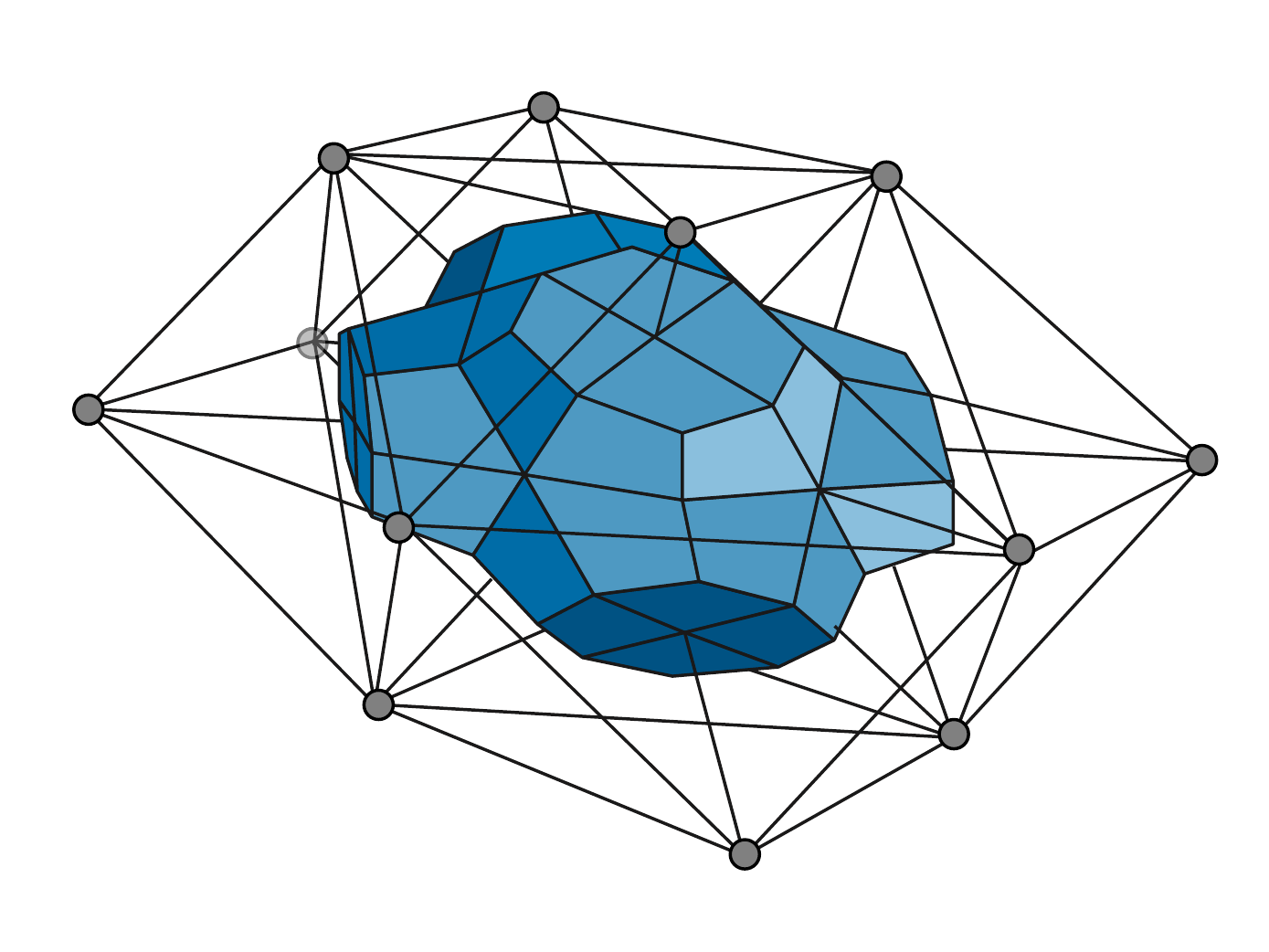}
    \caption{\label{fig:dualcells} Illustration of a nodal patch $\mathcal{T}_i$ around a vertex $\mathbf{x}_i$ with the associated dual control volume $B_i$ (shaded) for $d=2$ (left) and $d=3$ (right).}
\end{figure}
The boundary $\partial B_i$ of every such volume $B_i$ can be decomposed into the set of internal control facets $\gamma_{ij}^T \coloneqq \partial B_i \cap \partial B_j \cap T : B_j \ne B_i$ and boundary facets $\gamma_{i\partial}^T \coloneqq \partial\Omega \cap \partial B_i \cap \partial T$, where $B_i$ and $B_j \in \mathcal{B}_h$ are the control volumes defined with respect to the two vertices $\mathbf{x}_i$ and $\mathbf{x}_j$, which are both end points of an edge of some $T \in \mathcal{T}_h$. With each $\gamma_{ij}^T $ we associate a unique normal $\mathbf{n}_{ij}^T$ oriented from $B_i^T$ towards $B_j^T$. We note that the two faces $\gamma_{ij}^T $ and $\gamma_{ji}^T $ are geometrically the same but have a differently oriented normal.
%
We are now prepared to discuss mass conservation properties on the dual mesh.

\subsection{Corrected mass flux}
In the following, we consider a wider class of continuous-pressure discretization schemes. After we state the main result, we discuss simplifications that are possible if the choice of the finite-element spaces is restricted to stabilized linear equal-order elements as discussed in this work. For the moment let us put the previous definitions aside and assume that $\mathbf{V}_h \times Q_h$ is a continuous mixed finite element space which may be stable or unstable. For unstable choices, we require that appropriate stabilization terms are added to the variational form such that a unique discrete solution exists.


Given a finite element solution $(\mathbf{u}_h,p_h) \in \mathbf{V}_h \times Q_h$, we define a mass flux defect
\begin{equation} \label{eq:fluxc}
\kappa_{ij}^T(\mathbf{u}_h,p_h) := \frac{1}{(d+1)|\gamma_{ij}^T|} \big(\mathcal{R}_{i}^T(\mathbf{u}_h,p_h)-\mathcal{R}_j^T(\mathbf{u}_h,p_h)\big),
\end{equation} 

\noindent where the local residual fluxes $\mathcal{R}_i^T$ are defined as
\begin{align*}
\mathcal{R}_i^T(\mathbf{u}_h,p_h) := \int_{B_i^T} \div\mathbf{u}_h\,{\rm d}x - \int_T \div\mathbf{u}_h\,\phi_i\,{\rm d}x - c_T(p_h, \phi_i) - \langle \fst,\phi_i\rangle_T,
\end{align*}
\noindent with $\phi_i \in Q_h$ denoting the linear nodal shape function associated with the local vertex $i$ of element $T \in \mathcal{T}_h$. Finally, we define the corrected mass flux $j(\mathbf{u}_h,p_h)$ which is oriented and defined facet-wise as
\begin{equation}
\label{eq:flux-correction}
j(\mathbf{u}_h,p_h)|_{\gamma_{ij}^T} := \mathbf{u}_h |_{\gamma_{ij}^T} \cdot\mathbf{n}_{ij}^T - \kappa_{ij}^T(\mathbf{u}_h,p_h).
\end{equation}
\noindent Given these preparations, we formulate our main observation, which is the conservation of the corrected fluxes over the boundary of a dual control volume.
\begin{theorem} \label{thm:dual}
Let $(\mathbf{u}_h,p_h) \in \mathbf{V}_h \times Q_h$ denote the solution of the discrete Stokes problem and assume that $c_T (\cdot,1) + \langle \fst, 1 \rangle_T = 0$ for every $T \in \mathcal{T}_h$. Then the corrected mass-flux $j(\mathbf{u}_h,p_h)$ is locally conservative, i.e.,
$$
\int_{\partial B} j(\mathbf{u}_h,p_h) \,{\rm d}s = 0, \qquad B \in \mathcal{B}_h.
$$
\end{theorem}
\noindent The proof is based on summation arguments and the weak mass-conservation property of the discrete solution, see \cite[Theorem~3.2]{GmeWalWoh:2014}.

Revisiting the construction of the mass-correction, we observe that for a stable pairing without additional stabilization terms (such as the classical Taylor-Hood element), the defect $\kappa_{ij}^T(\mathbf{u}_h,p_h)$ obviously does not depend on $p_h$, while for stabilized equal-order linear elements the correction does not depend on $\mathbf{u}_h$ since we find by a simple integration that the divergence terms cancel inside each element. The remaining terms can often be considerably simplified, e.g., for the stabilized discretization we consider in this paper, we find that the flux correction assumes the simple form
\begin{equation} \label{eq:fluxPSPG}
\kappa^T_{ij}(p_h) = \delta_T h_T^2 \nabla p_h|_T \cdot \mathbf{n}_{ij}^T.
\end{equation}
\noindent which can be straightforwardly evaluated.

Given a locally conservative mass-flux on the dual mesh, we could next solve local equilibration problems to obtain local fluxes also with respect to the primal mesh. This allows the lifting of the correction into $H(\mathrm{div})$-conforming finite element spaces and thereby allows us to define a strongly mass-conservative velocity  solution with the same order of convergence as the original finite element solution without post-processing. Details are outlined in \cite{GmeWalWoh:2014}.

However, when coupling the equal-order linear stabilized scheme to transport problems, we can even circumvent the reconstruction step by directly solving the transport equations on the dual mesh in a finite volume fashion, using the corrected mass fluxes for advection.

\subsection{Conservative coupling to transport equations}
In many applications non-isothermal models (e.g., mantle convection) or reactive models (e.g., the simulation of thrombus formation in the human blood-stream) are considered. What these models have in common is that we couple fluid flow in the form of \eqref{stokes} to a (set of) transport equation(s) of the form:
\begin{equation}
\label{eq:energy}
\partial_t \vartheta + \div(\vartheta \mathbf{u} - \varepsilon \nabla \vartheta) = 0\qquad \text{ in } \Omega \times (0, t_{\rm end}],
\end{equation}
\noindent for some simulation end time $t_{\rm end} > 0$, where $\vartheta$ denotes the transported quantity, $\varepsilon > 0$ plays the role of a conductivity,
 and $\mathbf{u}$ is an often dominant advective velocity field that typically results from a flow simulation. To make the problem well-posed, we consider an initial temperature field $\vartheta(t=0)=\vartheta^0$ and homogeneous Neumann boundary conditions $\varepsilon \nabla\vartheta\cdot\mathbf{n}=0$ on the whole boundary $\partial\Omega$ of the computational domain, where $\mathbf{n}$ denotes an outward-pointing unit-normal vector. This is done only for simplicity of the exposition and as we will see in the numerical examples  the extension to other types of boundary conditions is possible.

Of course, in non-isothermal models the flow-properties such as the forcing or the viscosity may again depend non-linearly on $\vartheta$. However, for simplicity we shall consider only a one-way coupling between flow and transport here. For more general models with full non-isothermal coupling we refer to our recent work~\cite{waluga2015mass}.

Our aim is to complement the equal-order linear stabilized scheme with a finite-volume solver that associates nodal degrees of freedom with unknown coefficients. Therefore,  we define the following space on the barycentric dual mesh:
\begin{align*}
R_h := \{r_h \in L^2(\Omega)\, :\, r_h|_{B} \in P_0(B),\, B\in\mathcal{B}_h\},
\end{align*}
\noindent We note that due to the fact that $\mathrm{dim}\,R_h = \mathrm{dim}\,V_h = n_p$, we can define a natural bijection between these two spaces. More precisely, given the nodal basis $\phi_i$ of $V_h$ and the piecewise constant basis $\chi_i$ of $R_h$, we can define the natural transfer operators  $\pi_h: \sum_{i=1}^{n_p} q_i \chi_i \to \sum\noindent_{i=1}^{n_p} q_i \phi_i$, and $\overline{\pi}_h: \sum\noindent_{i=1}^{n_p} q_i \phi_i \to \sum\noindent_{i=1}^{n_p} q_i \chi_i$ which by construction satisfy the property $\pi_h\overline{\pi}_h = \overline{\pi}_h\pi_h = \text{id}$.

Based on the discrete velocity solution of \eqref{vastab} or \eqref{modvastab}, the transport equation for the temperature can be then advanced by an upwind finite-volume scheme which we state in form of a Petrov-Galerkin method: given $(\mathbf{u}_h^n, p_h^n) \in \mathbf{V}_h \times Q_h$ and $\vartheta_h^n \in V_h$, find $\vartheta_h^{n+1} \in V_h$ such that
\begin{align}
\label{eq:discrete-weak-form-transport}
\sum_{B_i \in \mathcal{B}_h}\int_{B_i}  \overline{\pi}_h \vartheta_h^{n+1}\,r_h \,{\rm d}x = \sum_{B_i \in \mathcal{B}_h}\int_{B_i} \vartheta_h^n\,r_h \,{\rm d}x - \Delta{t}_n\,\mathcal{A}(\vartheta_h^n, r_h),
\end{align}
\noindent for all $r_h \in R_h$, where the use of $\overline{\pi}_h$ introduces a mass-lumping, and we set the bilinear form
$$
\mathcal{A}(\vartheta_h^n, r_h) := \sum_{B_i \in \mathcal{B}_h} \int_{\partial B_i \backslash \partial\Omega} \left({\rm j}(\mathbf{u}_h^n, p_h^n)\,\langle\vartheta_h^n\rangle_{\rm up} - \varepsilon\nabla\vartheta_h^n\cdot\mathbf{n}_i\right)\,r_h\,{\rm d}s,
$$
\noindent where $\mathbf{n}_i$ is the unit normal vector pointing outward of $B_i$. To ensure stability, we use an up-winding, namely
\begin{align*}
\langle\vartheta_h^n\rangle_{\rm up}|_{\partial B_i\cap\gamma_{ij}^T} :=
\begin{cases}
\vartheta_h^n|_{B_i} & \text{if } {\rm J}_{ij}^T \ge 0\\
\vartheta_h^n|_{B_j} & \text{if } {\rm J}_{ij}^T \le 0\\
\end{cases},\qquad i \ne j,
\end{align*}
\noindent where we define the net flux ${\rm J}_{ij}^T := \int_{\gamma_{ij}^T}{\rm j}(\mathbf{u}_h^n, p_h^n)|_{\partial B_i}\,{\rm d}s$ based on the locally conservative facet flux function
\begin{align}
\label{eq:conservative-flux-pspg}
{\rm j}(\mathbf{u}_h^n, p_h^n)|_{\partial B_i} := (\mathbf{u}_h^n - \delta_T h_T^2 \nabla p_h^n)|_{\partial B_i}\cdot \mathbf{n}_i.
\end{align}

The distinctive feature of this method is that it can be implemented with minimal overhead, using only nodal data structures on semi-structured or even fully unstructured simplicial meshes. This collocated data layout makes the scheme particularly interesting for the implementation into highly performant stencil-based implementations and legacy codes. Of course, the construction above assumes that the discrete Stokes problem is solved exactly, which is often not feasible in practice, e.g., when iterative solvers are employed. However, the mass-defect introduced by solver inaccuracies are only related to the iteration error in contrast to the discretization error as it is common for methods that do not obey local conservation properties. 

\subsection{Numerical examples}
In the following we would like to show two examples in which the effect of the mass-corrections can be clearly observed. Firstly, we show a geometry in which flow is accelerated and decelerated, leading to regions with poor mass-conservation. We transport a concentration through this domain and show solutions obtained by the corrected coupling scheme and a non-corrected approach. Secondly, we discuss a more involved example from \cite{waluga2015mass} in which we consider non-isothermal fluid properties and induce a forcing by buoyancy terms.

\subsubsection{Pipe with obstacles}

As a simple example, we consider a pipe which is cylindrical with unit radius and aligned with the $z$-axis. Inside we place three spherical obstacles as shown in Fig.~\ref{fig:pipe-example-setup}. We further neglect forcing terms and diffusion, i.e., we set $\mathbf{f}=\mathbf{0}$ and $\varepsilon=0$. The inflow at the left boundary ($z=0$) is prescribed as $\mathbf{g} = (0,0,1-x^2-y^2)^\top$, at the right boundary we consider free outflow conditions, and for the remaining boundary we impose a no-slip condition. The transported quantity $\vartheta$ is set to $\vartheta=1$ at the left boundary and can exit the pipe domain freely at the right boundary.

\begin{figure}[hpbt]
\vspace*{-0.25cm}
\centering
\includegraphics[width=.55\textwidth]{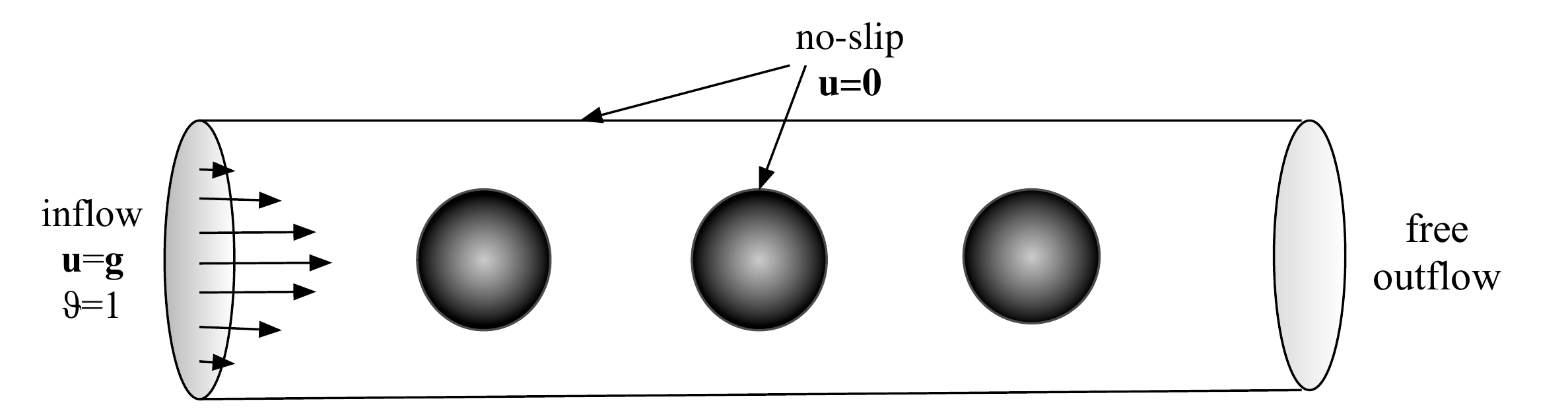}
\includegraphics[width=.415\textwidth]{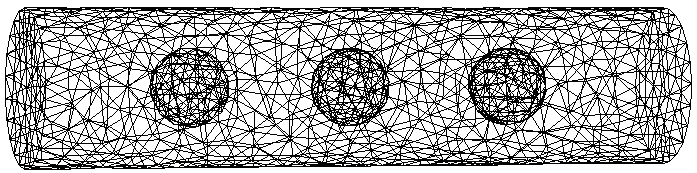}
    \caption{\label{fig:pipe-example-setup} Left: illustration of the geometry and the boundary conditions. Right: coarsest mesh.
}
\end{figure}

We conduct a mesh-refinement study on a series of three meshes with increasing resolution (cf. Fig.~\ref{fig:pipe-example-setup}) in which we choose our time-steps according to the CFL stability condition. In Fig.~\ref{fig:pipe-example}, we compare our results at $t_{\rm end}=10$ for a non-conservative scheme as well as for the same series of experiments conducted with the conservatively coupled scheme. It can be observed that the flux-correction keeps the concentration profiles within physical bounds, and already a visual inspection reveals that it produces physically better solutions with less resolution than the uncorrected method. In the non-conservative case, we observe solution overshoots of around $100\%$ even for the finest mesh. The application of a simple limiter to these overshoots would take mass/energy out of the system, and thus we cannot expect reasonable results when using moderate resolutions with a straightforward coupling. The conservative mass-correction is a simple and elegant way to overcome this deficiency.
\begin{figure}[hpbt]
\centering
\begin{minipage}{.475\textwidth}
	\includegraphics[width=\textwidth]{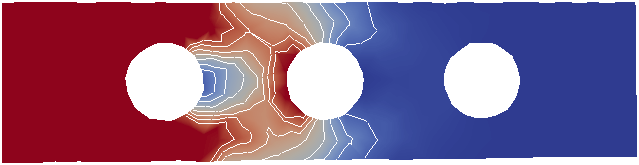}
	\includegraphics[width=\textwidth]{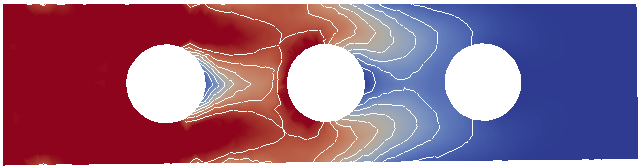}
	\includegraphics[width=\textwidth]{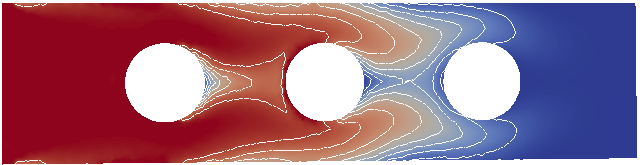}
\end{minipage}
\hfill
\begin{minipage}{.475\textwidth}
	\includegraphics[width=\textwidth]{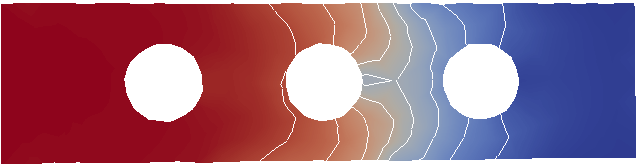}
	\includegraphics[width=\textwidth]{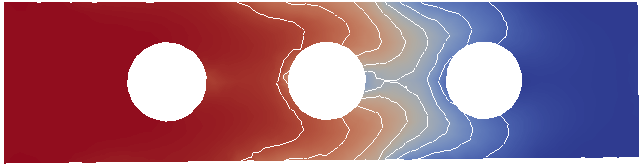}
	\includegraphics[width=\textwidth]{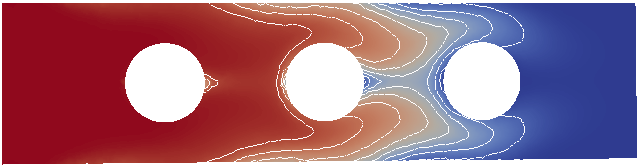}
\end{minipage}

    \caption{\label{fig:pipe-example}
    Cross-section in the $y/z$-plane of the computed solution on mesh-levels 0, 1, and 2 for the uncorrected scheme (left) and the corrected scheme (right); the plotting range is limited to $[0,1]$ in the left picture.
    }
\end{figure}

\subsubsection{Non-isoviscous case}
In the following, we  consider two examples for the non-isoviscous case, both regarding simulations of the Earth mantle. Firstly, the full non-isothermal coupling of buoyancy-driven flow models, where we have a non-linear bidirectional coupling between flow and transport. In particular for this problem, the conservative coupling is crucial in order to avoid a spurious forcing on the velocity field as well as non-physical values of the  temperature-dependent viscosity. Secondly, the simulation of the velocity field in the Earth mantle, where the boundary velocities for the Dirichlet condition on the surface and the temperature profile for the right-hand side are given by measured data. On the inner boundary,
 we use free-slip, i.e., $\mathbf{u} \cdot \mathbf{n} = (\nu D(\mathbf{u})\, \mathbf{n} - p\, \mathbf{n}) \cdot \mathbf{t}_i = 0$ where $\mathbf{t}_i$, $i=1,2$, denote the tangential vectors.

In these cases the momentum equation is given by
\begin{align}
-\mathrm{div } (\nu D(\mathbf{u})) + \nabla p = \mathbf{f},
\end{align}
where $D(\mathbf{u}) = 1/2( \nabla \mathbf{u} + \nabla \mathbf{u}^\top)$, denotes the symmetric part of the velocity gradient and $\mathbf{f}$ a temperature dependent forcing term. We consider in the following the spherical shell $\Omega = \{ x \in \mathbb{R}^3 : 0.55 < ||x||_2 < 1\}$ as a computational domain with two different viscosities, either temperature depending or piecewise constant
\begin{align}
\nu_1 = e^{1/2 - \vartheta},\qquad \nu_2 =
\begin{cases}
1 \quad& ||x||_2 < 0.936,\\
0.001 & \text{else}.
\end{cases}
\end{align}
For both examples the computational mesh is an icosahedral spherical mesh consisting of roughly a hundred million individual tetrahedra. For the first example with $\nu_1$, we consider thermal convection  where the interior temperature at the boundary is set to $\vartheta(r=0.55)=1$ and at the surface to $\vartheta(r=1)=0$. Initially we prescribe an interpolation between core and surface temperature with a spherical harmonic perturbation. We solve the coupled problem with the previously described SCG multigrid solver for the Stokes part, coupled to an explicit time-integration of the finite-volume transport scheme with flux-correction for 11\,500 time-steps on a mid-sized department cluster. A plot of a temperature iso-surface is depicted in Fig.~\ref{fig:non-isoviscous-icosphere-example} (left). For further details on the setup and the simulation parameters we refer to \cite{waluga2015mass}. For the second example with the discontinuous viscosity $\nu_2$, where the jump roughly appears in the upper mantle (asthenosphere), we consider as a solver the Uzawa multigrid method, described in Sect.~\ref{subsec:multigrid}. A plot of the simulation result is depicted in Fig.~\ref{fig:non-isoviscous-icosphere-example} (right) which shows the typical behavior of the velocity field and  the region of the viscosity jump.

\begin{figure}[hpbt]
	\centering
	\includegraphics[width=.44\textwidth]{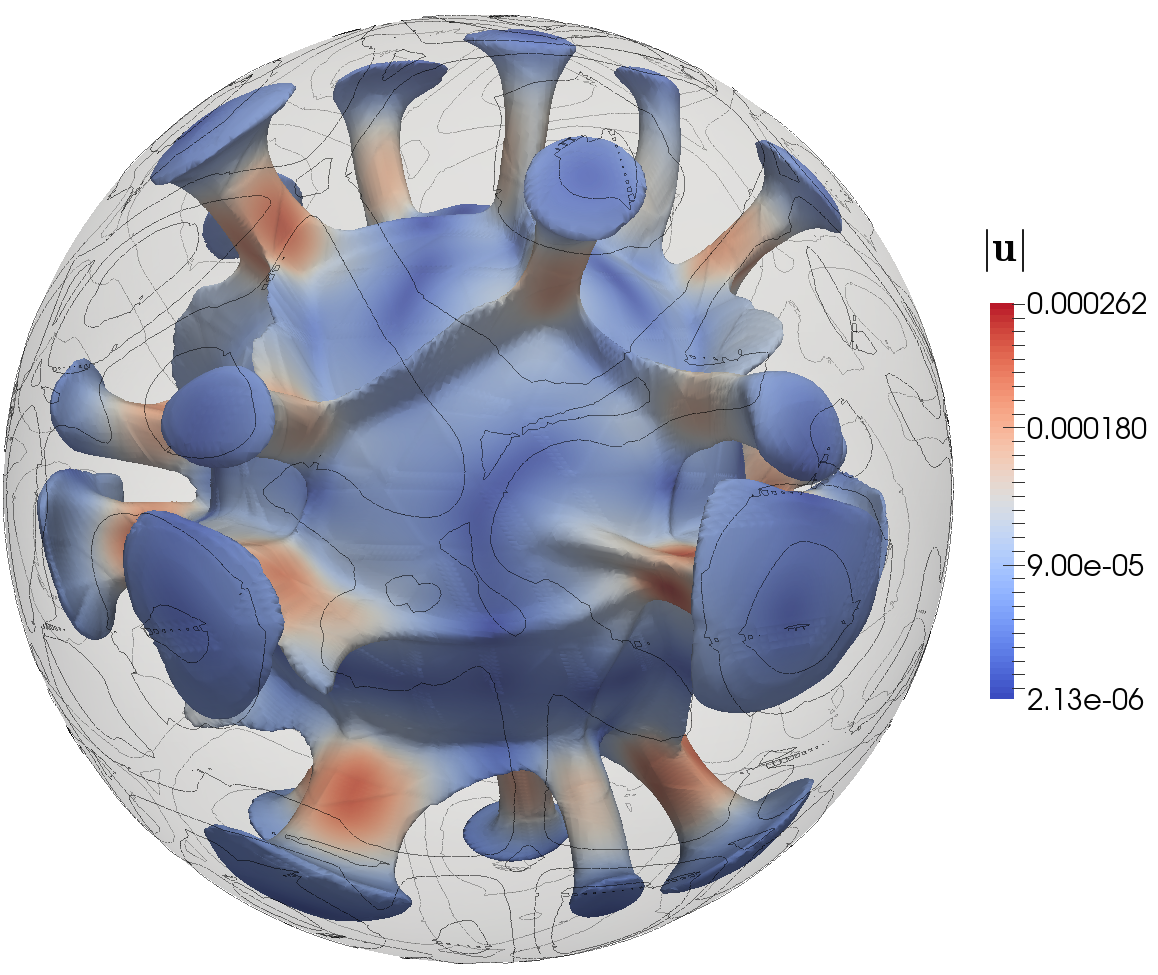}
	\hspace{0.5cm}
	\includegraphics[width=.45\textwidth]{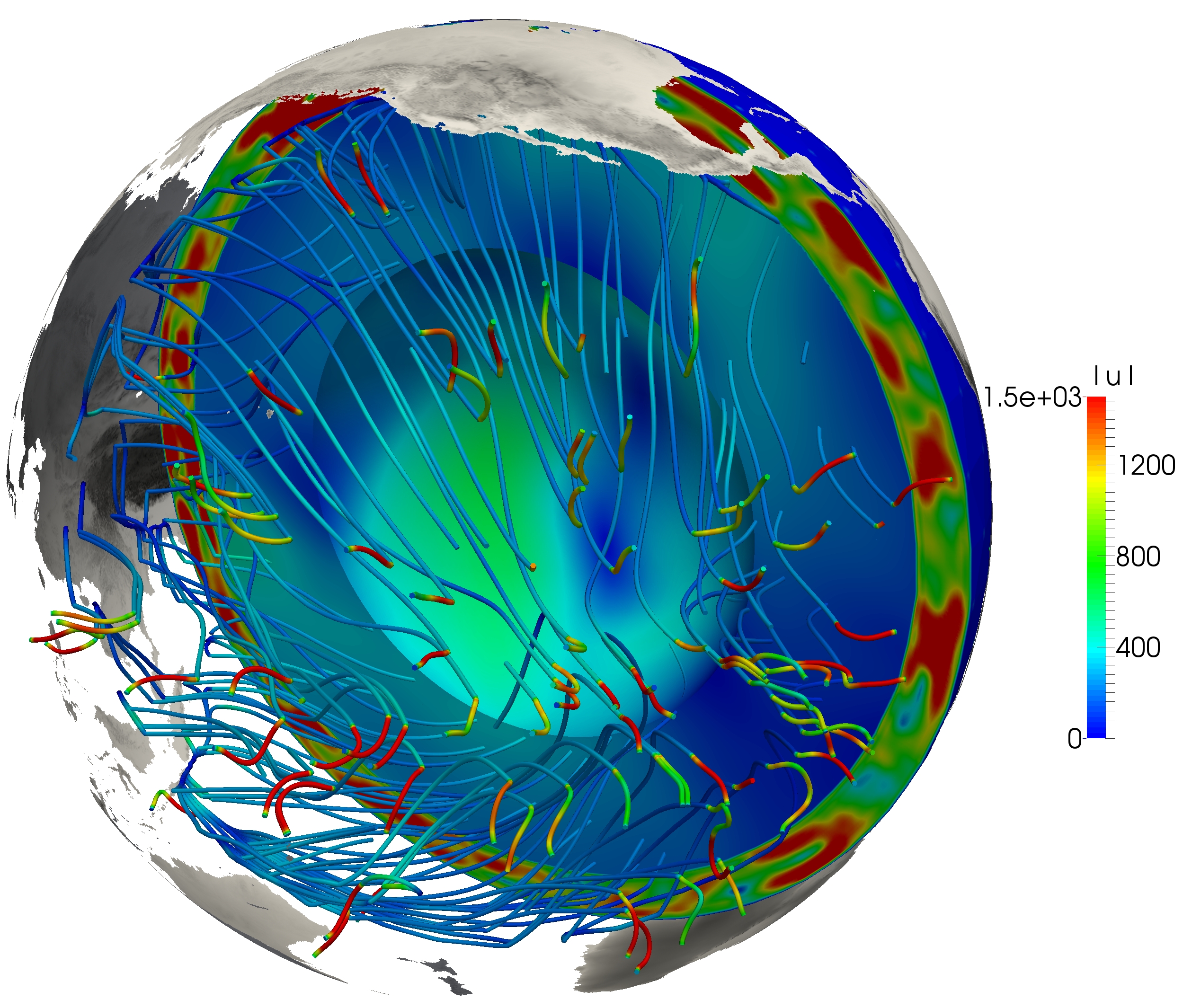}
    \caption{\label{fig:non-isoviscous-icosphere-example} Iso-surfaces for $\vartheta=0.4$ colored by the velocity magnitude after 11\,500 timesteps for $\nu_1$ (left) and magnitude of the velocity $\mathbf{u}$ for the real data simulation, including a viscosity jump $\nu_2$ (right).
}
\end{figure}


\vspace*{-0.5cm}
\paragraph{Acknowledgements}
The authors gratefully acknowledge the support by the Gauss Center for Supercomputing (GCS), the J\"ulich Supercomputing Centre (JSC), and the German Research Foundation (DFG) through the grants WO 671/14-1 and WO 671/11-1.


\vspace*{-0.5cm}
\frenchspacing
\begin{thebibliography}{7}

\bibitem{adams-fournier_book}
R.~A. Adams and J.~J.~F. Fournier.
\newblock {\em {Sobolev Spaces}}.
\newblock Academic Press, New York, London, 2003.

\bibitem{Arnold1984}
D.~N. Arnold, F.~Brezzi, and M.~Fortin.
\newblock A stable finite element for the {S}tokes equations.
\newblock {\em Calcolo}, 21(4):337--344, 1984.

\bibitem{bank-welfert-yserentant_1990}
R.~E. Bank, B.~D. Welfert, and H.~Yserentant.
\newblock A class of iterative methods for solving saddle point problems.
\newblock {\em Numer. Math.}, 56(7):645--666, 1990.

\bibitem{bergen-gradl-ruede-huelsemann_2006}
B.~Bergen, T.~Gradl, U.~R{\"u}de, and F.~H{\"u}lsemann.
\newblock A massively parallel multigrid method for finite elements.
\newblock {\em Computing in Science and Engineering}, 8(6):56--62, 2006.

\bibitem{bergen-huelsemann_2004}
B.~Bergen and F.~H{\"u}lsemann.
\newblock {Hierarchical hybrid grids: data structures and core algorithms for
  multigrid}.
\newblock {\em Numer. Linear Algebra Appl.}, 11:279--291, 2004.

\bibitem{Blum1990}
H.~Blum.
\newblock The influence of reentrant corners in the numerical approximation of
  viscous flow problems.
\newblock In {\em Numerical treatment of the {N}avier-{S}tokes equations
  ({K}iel, 1989)}, volume~30 of {\em Notes Numer. Fluid Mech.}, pages 37--46.
  Vieweg, Braunschweig, 1990.

\bibitem{bochev-dohrmann_2004}
P.~B. Bochev and C.~R. Dohrmann.
\newblock {A stabilized finite element method for the Stokes problem based on
  polynomial pressure projections.}
\newblock {\em {Internat. J. Numer. Methods Fluids}}, {46}(2):183--201, 2004.

\bibitem{Brezzi2013}
D.~Boffi, F.~Brezzi, and M.~Fortin.
\newblock {\em Mixed finite element methods and applications}, volume~44 of
  {\em Springer Series in Computational Mathematics}.
\newblock Springer, Heidelberg, 2013.

\bibitem{brezzi-bristeau-franca-mallet-roge_1992}
F.~Brezzi, M.~O. Bristeau, L.~P. Franca, M.~Mallet, and G.~Rog{\'e}.
\newblock A relationship between stabilized finite element methods and the
  {G}alerkin method with bubble functions.
\newblock {\em Comput. Methods Appl. Mech. Engrg.}, 96(1):117--129, 1992.

\bibitem{brezzi-douglas_1988}
F.~Brezzi and J.~Douglas, Jr.
\newblock Stabilized mixed methods for the {S}tokes problem.
\newblock {\em Numer. Math.}, 53(1-2):225--235, 1988.

\bibitem{Dauge1989}
M.~Dauge.
\newblock Stationary {S}tokes and {N}avier-{S}tokes systems on two- or
  three-dimensional domains with corners. {I}. {L}inearized equations.
\newblock {\em SIAM J. Math. Anal.}, 20(1):74--97, 1989.

\bibitem{Egger2014}
H.~Egger, U.~R{\"u}de, and B.~Wohlmuth.
\newblock Energy-corrected finite element methods for corner singularities.
\newblock {\em SIAM J. Numer. Anal.}, 52(1):171--193, 2014.

\bibitem{elman-silvester-wathen_2005}
H.~C. Elman, D.~J. Silvester, and A.~J. Wathen.
\newblock {\em {Finite elements and fast iterative solvers: with applications
  in incompressible fluid dynamics.}}
\newblock Oxford University Press, New York, 2005.

\bibitem{falgout-jones-yang_2006}
R.~D. Falgout, J.~E. Jones, and U.~M. Yang.
\newblock The design and implementation of hypre, a library of parallel high
  performance preconditioners.
\newblock In {\em Numerical solution of partial differential equations on
  parallel computers}, volume~51 of {\em Lect. Notes Comput. Sci. Eng.}, pages
  267--294. Springer, Berlin, 2006.

\bibitem{ganesan-matthies-tobiska_2008}
S.~Ganesan, G.~Matthies, and L.~Tobiska.
\newblock {Local projection stabilization of equal order interpolation applied
  to the Stokes problem}.
\newblock {\em {Math. Comp.}}, 77(264):2039--2060, 2008.

\bibitem{Gaspar_2014}
F.~J. Gaspar, Y.~Notay, C.~W. Oosterlee, and C.~Rodrigo.
\newblock A simple and efficient segregated smoother for the discrete {S}tokes
  equations.
\newblock {\em SIAM J. Sci. Comput.}, 36(3):A1187--A1206, 2014.

\bibitem{Girault1986}
V.~Girault and P.~A. Raviart.
\newblock {\em {Finite Element Methods for Navier-Stokes Equations.}}
\newblock Springer, New York, 1986.

\bibitem{gmeiner-ruede-stengel-waluga-wohlmuth_2015}
B.~Gmeiner, U.~R{\"u}de, H.~Stengel, C.~Waluga, and B.~Wohlmuth.
\newblock Performance and {S}calability of {H}ierarchical {H}ybrid {M}ultigrid
  {S}olvers for {S}tokes {S}ystems.
\newblock {\em SIAM J. Sci. Comput.}, 37(2):C143--C168, 2015.

\bibitem{gmeiner-ruede-stengle-waluga-wohlmuth_2015_2}
B.~Gmeiner, U.~R{\"u}de, H.~Stengel, C.~Waluga, and B.~Wohlmuth.
\newblock Towards textbook efficiency for parallel multigrid.
\newblock {\em Numer. Math. Theory Methods Appl.}, 8, 2015.

\bibitem{GmeWalWoh:2014}
B.~Gmeiner, C.~Waluga, and B.~Wohlmuth.
\newblock Local mass-corrections for continuous pressure approximations of
  incompressible flow.
\newblock {\em SIAM J. Numer. Anal.}, 52(6):2931--2956, 2014.

\bibitem{Guo2006}
B.~Guo and C.~Schwab.
\newblock Analytic regularity of {S}tokes flow on polygonal domains in
  countably weighted {S}obolev spaces.
\newblock {\em J. Comput. Appl. Math.}, 190(1-2):487--519, 2006.

\bibitem{pardiso2}
M.~Hagemann, O.~Schenk, and A.~W{\"a}chter.
\newblock {Matching--based preprocessing algorithms to the solution of
  saddle--point problems in large--scale nonconvex interior--point
  optimization.}
\newblock {\em {Comput. Optim. Appl.}}, {36}(2--3):321--341, 2007.

\bibitem{huber-gmeiner-ruede-wohlmuth_2015}
M.~Huber, B.~Gmeiner, U.~R{\"u}de, and B.~Wohlmuth.
\newblock Resilience for mutlgrid software at the extreme scale, 2015.
\newblock Submitted.

\bibitem{John2015}
L.~John, P.~Pustejovska, U.~R{\"u}de, and B.~Wohlmuth.
\newblock Energy-corrected finite element methods for the {S}tokes system,
  2015.
\newblock In preparation.

\bibitem{john-wohlmuth-zulehner}
L.~John, B.~Wohlmuth, and W.~Zulehner.
\newblock {On the Smoothing Properties of Block Preconditioners for Saddle
  Point Problems}, 2015.
\newblock In preperation.

\bibitem{Kondratjev1967}
V.~A. Kondratiev.
\newblock Boundary value problems for elliptic equations in domains with
  conical or angular points.
\newblock {\em Trans. Moscow Math. Soc.}, 16:227--313, 1967.

\bibitem{Kozlov2001}
V.~A. Kozlov, V.~G. Mazya, and J.~Rossmann.
\newblock {\em Spectral Problems Associated with Corner Singularities of
  Solutions to Elliptic Equations}.
\newblock Mathematical surveys and monographs. American Mathematical Society,
  2001.

\bibitem{kufner_1985}
A.~Kufner.
\newblock {\em Weighted {S}obolev spaces}.
\newblock A Wiley-Interscience Publication. John Wiley \& Sons, Inc., New York,
  1985.
\newblock Translated from the Czech.

\bibitem{Markl2008}
P.~M\"arkl and A.-M. S\"andig.
\newblock Singularities of the stokes system in polygons.
\newblock {\em Preprint-Reihe: Institut f{\"u}r Angewandte Analysis und
  Numerische Simulation, Universit\"at Stuttgart}, 2008/009, 2008.

\bibitem{Ruede2014}
U.~R{\"u}de, C.~Waluga, and B.~Wohlmuth.
\newblock Nested {N}ewton strategies for energy-corrected finite element
  methods.
\newblock {\em SIAM J. Sci. Comput.}, 36(4):A1359--A1383, 2014.

\bibitem{schoeberl-zulehner_2003}
J.~Sch{\"o}berl and W.~Zulehner.
\newblock On {S}chwarz-type smoothers for saddle point problems.
\newblock {\em Numer. Math.}, 95(2):377--399, 2003.

\bibitem{verfuerth_1984}
R.~Verf{\"u}rth.
\newblock A combined conjugate gradient-multigrid algorithm for the numerical
  solution of the {S}tokes problem.
\newblock {\em IMA J. Numer. Anal.}, 4(4):441--455, 1984.

\bibitem{waluga2015mass}
C.~Waluga, B.~Wohlmuth, and U.~R{\"u}de.
\newblock Mass-corrections for the conservative coupling of flow and transport
  on collocated meshes, 2015.
\newblock Submitted.

\bibitem{zulehner_2002}
W.~Zulehner.
\newblock Analysis of iterative methods for saddle point problems: a unified
  approach.
\newblock {\em Math. Comp.}, 71(238):479--505, 2002.

\end{thebibliography}
\end{document}